\documentclass[reqno]{amsart}

\calclayout

\usepackage[unicode]{hyperref}
\usepackage[lite]{amsrefs}
\usepackage{amsmath,amssymb}
\let\myemptyset\emptyset
\usepackage{graphicx}
\usepackage{subcaption}		
\usepackage{xspace}

\def\paragraph{\subsection}
\makeatletter
\def\@sect#1#2#3#4#5#6[#7]#8{%
  \edef\@toclevel{\ifnum#2=\@m 0\else\number#2\fi}%
  \ifnum #2>\c@secnumdepth \let\@secnumber\@empty
  \else \@xp\let\@xp\@secnumber\csname the#1\endcsname\fi
  \@tempskipa #5\relax
  \ifnum #2>\c@secnumdepth
    \let\@svsec\@empty
  \else
    \refstepcounter{#1}%
    \edef\@secnumpunct{%
      \ifdim\@tempskipa>\z@ 
        \@ifnotempty{#8}{.\@nx\enspace}%
      \else
        \@ifempty{#8}{.}{.\@nx\enspace}%
      \fi
    }%
      \ifnum #2=\tw@ \def\@secnumfont{\bfseries}\fi{}%
    \protected@edef\@svsec{%
      \ifnum#2<\@m
        \@ifundefined{#1name}{}{%
          \ignorespaces\csname #1name\endcsname\space
        }%
      \fi
      \@seccntformat{#1}%
    }%
  \fi
  \ifdim \@tempskipa>\z@ 
    \begingroup #6\relax
    \@hangfrom{\hskip #3\relax\@svsec}{\interlinepenalty\@M #8\par}%
    \endgroup
    \ifnum#2>\@m \else \@tocwrite{#1}{#8}\fi
  \else
  \def\@svsechd{#6\hskip #3\@svsec
    \@ifnotempty{#8}{\ignorespaces#8\unskip
       \@addpunct.}%
    \ifnum#2>\@m \else \@tocwrite{#1}{#8}\fi
  }%
  \fi
  \global\@nobreaktrue
  \@xsect{#5}}
\makeatother

\newtheorem{theorem}{Theorem}[section]
\newtheorem{conjecture}[theorem]{Conjecture}
\newtheorem{corollary}[theorem]{Corollary}
\newtheorem{definition}[theorem]{Definition}
\newtheorem{example}[theorem]{Example}
\newtheorem{lemma}[theorem]{Lemma}

\newcommand{\interject}[1]{\noalign{\begin{quote}#1\end{quote}}}

\newcommand{\disconnect}{\leavevmode\par}

\makeatletter
\def\pxspace{\@ifnextchar.{\@}{.\@\xspace}}
\makeatother

\newcommand{\forwardref}[1]{\ref{#1}}

\newcommand{\simclass}{$\sim$-class\xspace}
\newcommand{\simclasses}{$\sim$-classes\xspace}

\newcommand{\mbbB}{\mathbb{B}}
\newcommand{\Bfin}{\mbbB_{\textup{fin}}}
\newcommand{\Bcofin}{\mbbB_{\textup{cofin}}}
\newcommand{\mbbN}{\mathbb{N}}
\newcommand{\mbbZ}{\mathbb{Z}}

\newcommand{\zerohat}{\hat{0}}
\newcommand{\onehat}{\hat{1}}

\newcommand{\scrD}{\mathcal{D}}
\newcommand{\scrF}{\mathcal{F}}
\newcommand{\scrI}{\mathcal{I}}
\newcommand{\scrL}{\mathcal{L}}
\newcommand{\scrP}{\mathcal{P}}
\newcommand{\PF}[1]{{{#1}\mathclose{\uparrow}}}
\newcommand{\PI}[1]{{{#1}\mathclose{\downarrow}}}

\newcommand{\Symdiff}{\mathop{\triangle}}
\newcommand{\doublewedge}{\mathop{\wedge\mkern-8mu\wedge}}
\newcommand{\doublevee}{\mathop{\vee\mkern-8mu\vee}}
\newcommand{\UM}{\doublewedge}
\newcommand{\UJ}{\doublevee}

\newcommand{\Raise}{\mathord{\Uparrow}}
\newcommand{\Lower}{\mathord{\Downarrow}}
\newcommand{\sslash}{\mathord{
  \mathchoice{/\mkern-5mu/}
    {/\mkern-4mu/}
    {/\mkern-5mu/}
    {/\mkern-5mu/}}}
\newcommand{\transpose}[2]%
{{#1}^+/{#1}^- \searrow {#2}^+/{#2}^-}
\DeclareMathOperator{\im}{im}

\newcommand{\mvert}{\,|\,}      

\begin{document}

\newcommand{\ipeFigZxZ}{1}
\newcommand{\ipeFigZxZwithFilters}{2}
\newcommand{\ipeFigZxZfilters}{3}
\newcommand{\ipeFigZxZprimes}{4}
\newcommand{\ipeFigZxZidealsA}{5}
\newcommand{\ipeFigZxZchain}{6}
\newcommand{\ipeFigSSchain}{7}
\newcommand{\ipeFigSkewStrip}{8}
\newcommand{\ipeFigPlait}{9}
\newcommand{\ipeFigReachable}{10}
\newcommand{\ipeFigReachMessy}{11}
\newcommand{\ipeFigLPI}{12}
\newcommand{\ipeFigLPIlf}{13}
\newcommand{\ipeFigPILPI}{14}

\title[Representation of locally-finite lattices]%
{An extension of Birkhoff's representation theorem to
locally-finite distributive lattices}
\author{Dale R. Worley}
\email{worley@alum.mit.edu}
\date{Oct 2, 2026} 

\begin{abstract}
Birkhoff's representation theorem identifies a finite distributive
lattice with the lattice of order ideals of its poset of
join-irreducible elements
This can be extended to identifying a
finitary distributive lattice with the lattice of finite order ideals
of its poset of join-irreducible elements.
Each of these
representations can be extended into a duality, where the
representation and reconstruction are
inverse constructions between the isomorphism classes of the selected class
of distributive lattices and the isomorphism classes of of the
selected class of posets.
Many lattices arising in combinatorics are
locally finite but not finitary (where ``locally-finite'' means all
intervals contain a finite number of elements).
The main result of
this paper is an analogous representation duality for locally-finite
distributive lattices:  A locally-finite distributive lattice is dual
to a \emph{pair}, consisting of a poset (the poset of prime filters of
the lattice) and a distinguished subset of the order ideals of the
poset.
The distinguished subset is (1) a sublattice that is a maximal collection of
order ideals whose mutual symmetric differences are finite, and (2)
for every element of the poset, the subset contains ideals that both
include and exclude the element.
An alternative duality is given where in the pairs, the subset of
order ideals is
replaced by a single order ideal, and where isomorphism on the
class of pairs is correspondingly made broader.
We further show that when the
Hasse diagram of the lattice of order ideals of a poset
is regarded as a graph, any dual lattice is exactly one connected
component of the graph.
\end{abstract}

\maketitle

\textit{The 2\textsuperscript{\,nd} version is corrected and
  completed.  The 3\textsuperscript{\,rd} version adds a converse to
  the representation theorem to make a duality, adds a comparison to
  the previous
  literature, discusses a disproved conjecture, adds future
  directions, and renumbers from the 2\textsuperscript{\,nd} version.}

\section{Introduction} \label{sec:intro}

Representation theorems play a central role in the study of
distributive lattices. The foundation is Birkhoff's classical
representation theorem for finite distributive lattices.

\begin{theorem}[Birkhoff \cite{Birk1937a}\cite{Birk1967a}*{\S III.3 p.\ 58}] \label{th:birk}
Any finite distributive lattice $\scrL$ is isomorphic to the
lattice of order ideals of the partial order of the join-irreducible
elements of $\scrL$.
\end{theorem}

The representation is that every element $x \in \scrL$ is mapped into the
order ideal (lower set)
consisting of the join-irreducible elements of $\scrL$ that
are $\leq x$.  The meet and join of $\scrL$ then correspond to
intersection and union of order ideals.

Birkhoff's theorem extends naturally beyond finite lattices,
and particularly to finitary lattices. (A lattice
is \emph{finitary} if every principal ideal is finite.)
The finitary distributive lattices have an analogous representation.

\begin{theorem} \cite{CrawDil1973a}*{ch.\ 5} \cite{Stan2012a}*{Prop. 3.4.3} \label{th:stanley}
Any \emph{finitary} distributive lattice $\scrL$ is isomorphic to the
lattice of \emph{finite} order ideals of the partial order of the
join-irreducible elements of $\scrL$.
\end{theorem}

The correspondence is analogous, with each every element $x \in \scrL$
is mapped into the
order ideal of join-irreducible elements of $\scrL$ that
are $\leq x$, which is necessarily finite.  Again, the meet and join
of ideals is defined to be intersection and union.

Each of these representations can be extended into a duality, where
the representation and reconstruction are
inverse constructions between the isomorphism classes of
the selected class of
distributive lattices and the isomorphism
classes of of the selected class of posets.
In the case of finite distributive lattices, the
corresponding class is finite posets and in the case of finitary
distributive lattices, the corresponding class is finitary posets.

Many lattices arising in combinatorics are locally finite but not
finitary.  (We use the sense of ``locally-finite'' which is that all
intervals in the lattice contain a finite number of elements.)  The
question of whether an analogous representation theorem can
be obtained for locally-finite distributive lattices seems not to have
been addressed in the literature.

Two obstacles to constructing such a representation
immediately present themselves.  One is that some
locally-finite distributive lattices have no join-irreducible elements
at all, eliminating the objects from which the above representations
are built.  A premier example of this is $\mbbZ \times \mbbZ$, as
shown in Fig.~\ref{fig:ZxZ}.
The second obstacle is that even when
sufficient join-irreducible elements are available, the lattice is
often represented by only a proper subset of the order ideals of those
elements, but unlike with finitary lattices, there is no
straightforward universal pattern for selecting this subset of ideals.
Often these ideals are the proper ideals, all the order ideals with
the exception of the empty ideal and the ideal of all join-irreducible
elements.  A typical example is shown in Fig.~\ref{fig:Plait-SS}.

The main result of this paper is an analogous representation theorem for
locally-finite distributive lattices.
We prove that such a lattice is isomorphic to a distinguished sublattice of the
order ideals of the poset of prime filters of the lattice.
This sublattice is a maximal collection of order ideals whose
mutual symmetric differences are finite.
We further show that when the
Hasse diagram of the lattice of order ideals is regarded as a graph,
this sublattice is exactly one connected component.

This representation can be analogously extended into a
duality.  The dual class turns out to be \emph{pairs},
with one component of each pair being an (otherwise unrestricted) poset and the
other component being a distinguished set of order ideals of the
poset satisfying the restrictions of Def.~\forwardref{def:pair}.   The class of
pairs is considered under a novel definition of isomorphism:
an isomorphism between the posets of the pairs which maps the sets of
distinguished order ideals into each other, as described in
Def.~\forwardref{def:pair-isom}.

When discussing specific pairs, it is often more
convenient to specify just one of the order ideals of the
distinguished class rather
than specify the class.  This is embodied in an alternative form of
the duality, in which the second component of the pairs is a single
order ideal of the poset, and also isomorphism of pairs is revised so that
two pairs with the same poset
are isomorphic iff the symmetric difference of their two ideals is finite.

In the light of the previous representations, that such a
representation duality exists is not surprising.  However the specific
details of the representation are considerably more complex than in
the previous versions, and in particular it is novel that the pair
must have a second component to select how the lattice is
reconstructed from the poset.  Similarly to the previous
representations, this representation provides a fully general tool for
examining its class of lattices.

The paper is organized as follows:
Sec.~\ref{sec:general} recalls Nation's \cite{Nat2017a} variant of
Stone's \cite{Stone1938a} representation of general distributive
lattices based on prime filters.
Sec.~\ref{sec:finitary} provides a summary of the
representation of finitary lattices.
Sec.~\ref{sec:tools} defines tools that will be used in the
novel representation.
Sec.~\ref{sec:locally-finite} presents the novel representation
duality for locally-finite distributive lattices.
Sec.~\ref{sec:criterion} gives a novel method for identifying which
prime filters are secondary, that is, not the principal filters for
join-irreducible elements of the base lattice.
Sec.~\ref{sec:unique} discusses a conjecture that was included in
the 2\textsuperscript{\,nd} version of this paper but that has since
been disproved.
Sec.~\ref{sec:future} discusses possible future directions for
research on this subject.
A final section describes the history of the development of this paper.

I am grateful to Prof.\ Nation for his advice on how to present these
results, and particularly for directing me to this paper's antecedents
in Crawley and Dilworth \cite{CrawDil1973a}
and Davey and Priestley \cite{DavPriest1990a}.

\begin{figure}[ht]
\centering
\begin{minipage}{0.45\textwidth}
  \centering
  \includegraphics[page=\ipeFigZxZ]{locally-finite-figs.pdf}%
  \caption{The lattice $\mbbZ \times \mbbZ$}
  \label{fig:ZxZ}
\end{minipage}
\begin{minipage}{0.45\textwidth}
  \centering
  \includegraphics[page=\ipeFigZxZwithFilters]{locally-finite-figs.pdf}%
  \caption{The lattice $\mbbZ \times \mbbZ$ with some filters outlined}
  \label{fig:ZxZwithFilters}
\end{minipage}
\end{figure}

\begin{figure}[ht]
\hbox to \linewidth{\hfill%
\includegraphics[page=\ipeFigSkewStrip,scale=0.75]{locally-finite-figs.pdf}%
\hfill
\includegraphics[page=\ipeFigPlait]{locally-finite-figs.pdf}%
\hfill}
\caption{The lattice $t$-SkewStrip (left) and its poset of
join-irreducible elements $t$-Plait (right), for $t = 3$.
(After Fomin \cite{Fom1994a}*{Fig.~7}.)}
\label{fig:Plait-SS}
\end{figure}

\section{Representation of general distributive lattices} \label{sec:general}

We now describe representation of an arbitrary distributive lattice
in terms of its prime filters. The construction follows
Nation \cite{Nat2017a}*{chs.~2, 8, and~10}, which is a variation on
Davey--Priestley
\cite{DavPriest1990a}*{chs.~8 and~10} and Stone \cite{Stone1938a}.
Our construction is presented in a more combinatorial form
that avoids topological machinery.
This combinatorial approach naturally accommodates unbounded
lattices. Its main drawback is that it does not provide as concise a
characterization of the representing ideals as the topological
formulation.

The connection to \cite{DavPriest1990a} is as follows:

Start with a distributive lattice $\scrL$.
Let $\scrF$ be the lattice of lattice filters of $\scrL$, ordered
by \emph{inverse} inclusion. Let $\scrP \subset \scrF$ be the subposet of
prime filters. Finally, let $\scrI$ be the lattice of order
ideals of $\scrP$.
Then we have this lattice embedding from $\scrL$ into $\scrI$:

$$ \phi(x) = \{p \in \scrP \mvert x \in p\} $$

Dually, we define the lattice $\scrF^\prime$ of
the lattice ideals of $\scrL$, ordered by \emph{inclusion}.  Then we
define $\scrP^\prime$ to be the sub-poset of
$\scrF^\prime$ of prime ideals, also ordered by inclusion.
We define $\scrI^\prime$ to be the lattice of
order ideals of $\scrP^\prime$, ordered by inclusion.

The parallel construction of \cite{DavPriest1990a}*{sec.~10.18} is the
map (in their notation) $X_\bullet$ from $L$ to $\scrI(\scrI_p(L))$.
Our $\scrL$ is their $L$,
our $\scrF^\prime$ is their $\scrI(L)$, our $\scrP^\prime$ is
their $\scrI_p(L)$, and our $\scrI^\prime$ is their $\scrI(\scrI_p(L))$.
They define:
$$ X_\bullet: L \rightarrow \scrI(\scrI_p(L))
: a \mapsto \{ I \in \scrI_p(L) \mvert a \not\in I \} $$
which is in our notation:
$$ X_\bullet: \scrL \rightarrow \scrI^\prime
: a \mapsto \{ p^\prime \in \scrP^\prime \mvert a \not\in p^\prime \} $$

Our Lem.~\forwardref{lem:P-Pprime} shows that
$\omega: x \mapsto \scrL \setminus x$ is a poset isomorphism between
$\scrP^\prime$ and $\scrP$, which naturally extends to a lattice
isomorphism $\omega_\#$ between $\scrI^\prime$ and $\scrI$.
$\omega_\#$ transforms $X_\bullet$ into $\phi$:
\begin{align*}
\omega_\#(X_x)
& = \omega_\# (\{ p^\prime \in \scrP^\prime \mvert x \not\in p^\prime \}) \\
& = \{ \omega(p^\prime) \mvert
  p^\prime \in \scrP^\prime \textup{ and } x \not\in p^\prime \} \\
\interject{Substituting $p = \omega(p^\prime)$ and
$p^\prime = \omega(p) = \scrL \setminus p$,}
& = \{ p \mvert p \in \scrP \textup{ and } x \not\in \scrL \setminus p \} \\
& = \{ p \in \scrP \mvert x \in p \} \\
& = \phi(x)
\end{align*}

Before proceeding, we introduce a running example that will be used
throughout this section: $\mbbZ\times\mbbZ$, shown in
Fig.~\ref{fig:ZxZ}.

For the remainder of this paper, let $\scrL$ be a distributive lattice.

\begin{definition} \label{def:ofilter}
Let $P$ be a poset.
We define that $F \subset P$ is an \emph{order filter} (of $P$) iff
$x \in F$ and $x \leq z$ imply $z \in F$.
We define that $I \subset P$ is an \emph{order ideal} (of $P$) iff
$x \in I$ and $z \leq x$ imply $z \in I$.
\end{definition}

\begin{definition} \cite{DavPriest1990a}*{8.7} \label{def:irred}
Let $\scrL$ be a lattice.
We define an element $x \in \scrL$ to be \emph{join-irreducible} iff
(1) for all $y, z \in \scrL$, $x = y \vee z$ implies $x = y$ or $x = z$; and
(2) if $\zerohat_\scrL$ exists, $x \neq \zerohat_\scrL$.

We define $x$ to be \emph{join-reducible} iff $x$ is not
join-irreducible.
Equivalently, $x$ is join-reducible iff it is equal to the join of a
finite number (if $\zerohat_\scrL$ exists, possibly zero) of
\emph{other} elements of $\scrL$.

We define \emph{meet-irreducible} and \emph{meet-reducible} dually.
\end{definition}

\begin{definition} \label{def:lfilter}
Let $\scrL$ be a lattice.
We define that $f \subset \scrL$ is a \emph{lattice filter} (of $\scrL$) iff
(1) $f$ is non-empty,
(2) $x,y \in f$ implies $x \wedge y \in f$, and
(3) $x \in f$ and $x \leq z$ imply $z \in f$.
We define that $i \subset \scrL$ is a \emph{lattice ideal} (of $\scrL$) iff
(1) $i$ is non-empty,
(2) $x,y \in i$ implies $x \vee y \in i$, and
(3) $x \in i$ and $z \leq x$ imply $z \in i$.
\end{definition}

Note that $\myemptyset$ is admitted as an order ideal and order filter
but not as a lattice ideal or lattice filter.

\begin{definition} \label{def:plfilter}
We define that $f$ is a \emph{proper lattice filter} (of a lattice $\scrL$) iff
$f$ is a lattice filter and $f$ is $\neq \scrL$.
We define a \emph{proper lattice ideal} dually:
$f$ is a lattice ideal and $f$ is $\neq \scrL$.
\end{definition}

There is of course a natural duality on lattices, so for many facts we
prove, we also state the dual fact without giving the dual proof.

\begin{definition} \label{def:F}
We define $\scrF$ to be the poset of lattice filters of
$\scrL$, ordered by \emph{inverse} inclusion.
Dually, we define $\scrF^\prime$ to be the poset of lattice ideals of
$\scrL$, ordered by \emph{inclusion}.
\end{definition}

For $\scrL = \mbbZ \times \mbbZ$, some filters (all proper) are shown in
Fig.~\ref{fig:ZxZwithFilters} and the corresponding lattice $\scrF$ is
shown in Fig.~\ref{fig:ZxZfilters}.

\begin{figure}[ht]
\centering
\begin{minipage}{0.45\textwidth}
  \centering
  \includegraphics[page=\ipeFigZxZfilters]{locally-finite-figs.pdf}%
  \caption{The lattice of filters $\scrF$ for $\scrL = \mbbZ
    \times \mbbZ$ with its prime elements outlined, all of which are
    secondary primes}
  \label{fig:ZxZfilters}
\end{minipage}
\begin{minipage}{0.45\textwidth}
  \centering
  \includegraphics[page=\ipeFigZxZprimes]{locally-finite-figs.pdf}%
  \caption{The poset of prime filters $\scrP$ for $\scrL = \mbbZ
    \times \mbbZ$.  $\scrP = A \sqcup B$}
  \label{fig:ZxZprimes}
\end{minipage}
\end{figure}

\begin{definition} \label{def:UM}
Given $f, g \in \scrF$, we define
$f \UM g = \{x \wedge y \mvert x \in f \textup{ and } y \in g\}$.
Dually, given $f^\prime, g^\prime \in \scrF^\prime$, we define
$f^\prime \UJ g^\prime =
\{x \vee y \mvert x \in f^\prime \textup{ and } y \in g^\prime\}$.%
\footnote{We pronounce $f \UM g$ as ``$f$ union-meet $g$'' due to
Lem.~\forwardref{lem:meet-closure}.
Dually, we pronounce $f^\prime \UJ g^\prime$ as
``$f^\prime$ union-join $g^\prime$''.}
\end{definition}

\begin{lemma} \label{lem:UM-contains}
Given $f, g \in \scrF$, $f \UM g$ contains $f$ and $g$.
Dually, given $f^\prime, g^\prime \in \scrF^\prime$,
$f^\prime \UJ g^\prime$ contains $f^\prime$ and $g^\prime$.
\end{lemma}
\begin{proof} \disconnect
Given any $x \in f$ we can choose any
$y\in g$ and then $x \vee y \in g$.  Then
$x = x \wedge (x \vee y)$, shows $x \in f \UM g$.
Similarly, $f \UM g$ contains $g$.
\end{proof}

\begin{lemma} \label{lem:meet-closure}
Given $f, g \in \scrF$,
$f \UM g$ is a lattice filter and is the meet-closure of $f \cup g$.
Dually, given $f^\prime, g^\prime \in \scrF^\prime$,
$f^\prime \UJ g^\prime$ is a lattice ideal and is the join-closure
of $f^\prime \cup g^\prime$.
\end{lemma}
\begin{proof} \disconnect
Regarding $f \UM g$ is non-empty:  This is trivial because $f$ and $g$
are non-empty.

Regarding $f \UM g$ is meet-closed:
Given two elements of $f \UM g$, $x \wedge y$ and
$x^\prime \wedge y^\prime$ with $x, x^\prime \in f$ and
$y, y^\prime \in g$, since $f$ and $g$ are meet-closed,
$x \wedge x^\prime \in f$ and $y \wedge y^\prime \in g$.
Then
$(x \wedge y) \wedge (x^\prime \wedge y^\prime)
= (x \wedge x^\prime) \wedge (y \wedge y^\prime) \in f \UM g$.

Regarding $f \UM g$ is closed upward:
Given $x \in f$, $y \in g$ and $z \in \scrL$ with
$x \wedge y \leq z$
then $z \vee x \in f$ and $z \vee y \in g$.
By distributivity, $(z \vee x) \wedge (z \vee y) = z \vee (x \wedge y) = z$,
showing $z \in f \UM g$.

Regarding $f \UM g$ is contained in the meet-closure of $f \cup g$:
Trivial by Def.~\ref{def:UM}.

Regarding $f \UM g$ contains the meet-closure of $f \cup g$:
This follows from Lem.~\ref{lem:UM-contains} and the fact that
$f \UM g$ is meet-closed.
\end{proof}

\begin{lemma} \label{lem:F-lattice}
$\scrF$ is a lattice, with its join being $\cap$ and its meet being $\UM$.
Dually,
$\scrF^\prime$ is a lattice, with its join being $\UJ$ and its meet
being $\cap$.
\end{lemma}
\begin{proof} \disconnect
Regarding join:
$f \cap g$ is clearly closed upward and closed under meets.
Since there exist $x \in f$ and $y \in g$, $x \vee y$ is in both $f$
and $g$ and so is in $f \cap g$, which is thus nonempty.
Together, these show that $f \cap g$ is a lattice filter.
$f \cap g$ is
$\geq$ (that is, is contained in) $f$ and $g$.  Clearly, any lattice filter
$h \geq f, g$ is contained in both $f$ and $g$, so $h$ must be contained
in $f \cap g$, and so must be $\geq f \cap g$.
Thus the join of two lattice filters $f$ and $g$ is $f \cap g$.

Regarding meet:
Given two lattice filters, $f$ and $g$, let $h = f \UM g$.
By Lem.~\ref{lem:UM-contains}, $f \subset h$, which is $h \leq f$.
Similarly $h \leq g$.
To show that $h$ is the greatest lower bound of $f$ and $g$,
assume $s$ is an lattice filter and $s \leq f, g$, that is,
$f \subset s$ and $g \subset s$.
Consider an arbitrary element $z$ of $h$,
so $z = x \wedge y$ for some $x \in f$ and $y \in g$.
Then $x$ and $y$ are $\in s$, and since
$s$ is closed under meets, $z = x \wedge y \in s$.
Thus $h \subset s$, which is $s \leq h$.
\end{proof}

\begin{lemma} \label{lem:distributive}
$\scrF$ is a distributive lattice.
Dually, $\scrF^\prime$ is a distributive lattice.
(In both cases, under the operations described in Lem.~\ref{lem:F-lattice}.)
\end{lemma}
\begin{proof} \disconnect
Since we know $\scrF$ is a lattice, it suffices to prove that
for any $f, g, h \in \scrF$,
$f \cap (g \UM h) = (f \cap g) \UM (f \cap h)$, which is equivalent to
showing that for any $x \in \scrL$,
$x \in f \cap (g \UM h)$ iff $x \in (f \cap g) \UM (f \cap h)$.
To show that, the following are equivalent:
\begin{gather}
x \in f \cap (g \UM h) \notag \\
x \in f \textup{ and } x \in g \UM h \notag \\
x \in f \textup{ and } (\exists y \in g, z \in h)\, x = y \wedge z
    \label{eq:dist-A}
\end{gather}
To derive (\forwardref{eq:dist-B}) from (\ref{eq:dist-A}):
Since $x = y \wedge z = x \vee (y \wedge z) = (x \vee y) \wedge (x \vee z)$
and $f$, $g$, and $h$ are lattice filters,
$$ x \in f \textup{ and }
(\exists y \in g, z \in h)\, x \vee y \in f \cap g \textup{ and }
x \vee z \in f \cap h \textup{ and }
x = (x \vee y) \wedge (x \vee z) $$
Then choose $ y^\prime = x \vee y$ and $z^\prime = x \vee z$.
To derive (\ref{eq:dist-A}) from (\forwardref{eq:dist-B}):
Choose $y = y^\prime$ and $z = z^\prime$, and
note that $f$ is closed under meets.
Thus, (\ref{eq:dist-A}) is equivalent to:
\begin{gather}
(\exists y^\prime \in f \cap g, z^\prime \in f \cap h)\,
x = y^\prime \wedge z^\prime \label{eq:dist-B} \\
x \in (f \cap g) \UM (f \cap h) \notag
\end{gather}
\end{proof}

\begin{definition} \label{def:pfilter}
We define that a lattice filter $f$ is \emph{prime} or is
a \emph{prime filter} iff
$f$ is proper and if $x \vee y \in f$ implies $x \in f$ or
$y \in f$.
We define a \emph{prime ideal} dually.
\end{definition}

\begin{lemma} \label{lem:filter-ideal}
A lattice filter $f$ is prime iff
$\scrL \setminus f$ is a lattice ideal iff
$\scrL \setminus f$ is a prime ideal.
Dually,
a lattice ideal $f$ is prime iff
$\scrL \setminus f$ is a lattice filter iff
$\scrL \setminus f$ is a prime filter.
\end{lemma}

\begin{definition} \label{def:PF}
For $x \in \scrL$, we define
$\PF{x} = \{y \in \scrL \mvert x \leq y\}$, the \emph{principal filter
generated by} $x$.
For $x \in \scrL$, we define
$\PI{x} = \{y \in \scrL \mvert y \leq x\}$, the \emph{principal ideal
generated by} $x$.
\end{definition}

\begin{lemma} \label{lem:PF-irred}
For $x \in \scrL$, $\PF{x}$ is a prime filter iff $x$ is
join-irreducible in $\scrL$.
Dually, for $x \in \scrL$, $\PI{x}$ is a prime ideal iff $x$ is
meet-irreducible in $\scrL$.
\end{lemma}

Note that if $\scrL$ were not distributive, the conditions in the
preceding lemma would have to be ``$x$ is join-prime'' and
``$x$ is meet-prime'' (using the convention that ``join-prime'' and
``meet-prime'' exclude $\zerohat$).

\begin{definition} \label{def:P}
We define $\scrP$ to be the poset of prime filters in $\scrF$, with
the order (inverse inclusion) inherited from $\scrF$.
Dually,
we define $\scrP^\prime$ to be the poset of prime ideals in
$\scrF^\prime$, with
the order (inclusion) inherited from $\scrF^\prime$.
\end{definition}

\begin{definition} \label{def:I}
We define $\scrI$ to be the lattice of order ideals of $\scrP$
under inclusion, which has $\cap$ as meet and $\cup$ as join.
Dually, we define $\scrI^\prime$ to be the lattice of order ideals of
$\scrP^\prime$
under inclusion, which has $\cap$ as meet and $\cup$ as join.
\end{definition}

For $\scrL = \mbbZ \times \mbbZ$, the corresponding lattice $\scrI$ is
shown in Fig.~\ref{fig:ZxZidealsA}.

\begin{figure}[ht]
\centering
  \includegraphics[page=\ipeFigZxZidealsA]{locally-finite-figs.pdf}%
  \caption{The lattice of ideals $\scrI$ of $\scrP$ for $\scrL = \mbbZ
    \times \mbbZ$ with its connected components outlined.}
  \label{fig:ZxZidealsA}
\end{figure}

\begin{lemma} \label{lem:P-Pprime}
Lem.~\ref{lem:filter-ideal} shows that the map
$\omega: \scrP \rightarrow \scrP^\prime: p \mapsto \scrL \setminus p$
is a poset isomorphism from $\scrP$ to $\scrP^\prime$.
Thus, its image map $\omega_\#$ is a lattice isomorphism from $\scrI$
to $\scrI^\prime$.
\end{lemma}

Note that because $\scrP$ is ordered by inverse inclusion and
$\scrP^\prime$ is ordered by inclusion, $\omega$ is an isomorphism,
not a dual isomorphism.

\begin{lemma} \label{lem:equiv}
Given $x \in \scrL$ and $f \in \scrF$, the following are equivalent:
\begin{enumerate}
\item $f \leq \PF{x}$,
\item $\PF{x} \subset f$, and
\item $x \in f$.
\end{enumerate}
Dually,
given $x \in \scrL$ and $f^\prime \in \scrF^\prime$, the following are
equivalent:
\begin{enumerate}
\item $\PI{x} \leq f^\prime$,
\item $\PI{x} \subset f^\prime$, and
\item $x \in f^\prime$.
\end{enumerate}
\end{lemma}

\begin{lemma} \label{lem:PF-embed}
$\PF{\bullet}: \scrL \rightarrow \scrF$ is a lattice embedding of
$\scrL$ into $\scrF$.
Dually,
$\PI{\bullet}: \scrL \rightarrow \scrF^\prime$ is a lattice embedding of
$\scrL$ into $\scrF^\prime$.
\end{lemma}

\begin{lemma} \label{lem:ji-prime}
$f \in \scrF$ is join-irreducible in $\scrF$ iff $f$ is a prime filter.
Dually,
$f^\prime \in \scrF^\prime$ is meet-irreducible in $\scrF^\prime$ iff
$f^\prime$ is a prime ideal.
\end{lemma}
\begin{proof} \disconnect
First we prove the lemma for filters.
Recall that $\scrF$ is ordered by reverse containment,
so $\zerohat_\scrF = \scrL$, its join
is $\cap$ and its meet is $\UM$.
We break the proof into two cases.
In the first case, $f = \zerohat$.
By Def.~\ref{def:irred}, $\zerohat$ is join-reducible, and by
Defs.~\ref{def:plfilter} and~\ref{def:pfilter},
$\zerohat_\scrF = \scrL$ is not a proper filter and thus is not a
prime filter.  Thus the lemma is true.

We will prove the remaining case by proving the contrapositive:
given $f \neq \zerohat$,
$f \in \scrF$ is join-reducible in $\scrF$ iff $f$ is not prime.

Regarding $\Rightarrow$:
Assume $f$ is join-reducible and $\neq \zerohat$.
Then there exist filters $g$ and $h$ for which
$g, h \neq f$ and $f = g \vee h = g \cap h$.
Choose $x \in g \setminus f$ and $y \in h \setminus f$; thus
$x, y \not\in f$.
Since $g$ and $h$ are filters, $x \vee y \in g, h$, and so
$x \vee y \in f$.  Thus $f$ is not prime.

Regarding $\Leftarrow$:
Assume $f$ is not prime and is a proper filter.
Then there exists $x$ and $y$ in $\scrL$ with
$x, y \not\in f$ but $x \vee y \in f$.
Define $g = \PF{x} \wedge f = \PF{x} \UM f$ and
$h = \PF{y} \wedge f = \PF{y} \UM f$;
then $g, h \leq f$ in $\scrF$.
But since $x \in g$ and $x \not\in f$, $g \neq f$, so $g < f$.
Similarly, since $y \in h$ and $y \not\in f$, $h \neq f$, so $h < f$.
We will show that
$f = g \vee h = g \cap h$.
By Lem.~\ref{lem:UM-contains}, $f \subset g \cap h$.
Conversely, suppose $z \in g \cap h$ (so $z \in \scrL$).
Then there exists $u, u^\prime, v, v^\prime \in \scrL$ with
$u \geq x$, $v \geq y$, and $u^\prime, v^\prime \in f$
such that $z = u \wedge u^\prime = v \wedge v^\prime$.
Since $\scrL$ is distributive,
$z = z \vee z = (u \wedge u^\prime) \vee (v \wedge v^\prime)
= (u \vee v) \wedge (u \vee v^\prime) \wedge (u^\prime \vee v)
\wedge (u^\prime \vee v^\prime)$.
But all four factors of that meet are $\in f$, so $z \in f$,
showing $g \cap h \subset f$.   
Thus, $f = g \cap h = g \vee h$, showing $f$ is join-reducible.

The proof for ideals is dual.  Recall that $\scrF^\prime$ is ordered
by containment, so $\onehat_{\scrF^\prime} = \scrL$,
its join is $\UJ$, and its meet is $\cap$.
The special case is when
$f^\prime = \onehat_{\scrF^\prime} = \scrL$, in which case $f^\prime$
is meet-reducible by Def.~\ref{def:irred}, and $f^\prime$ is not proper, so is
not a prime ideal.  The remaining case is proved dually to that
case for filters.
\end{proof}

Lem.~\ref{lem:ji-prime}, \ref{lem:PF-embed}, and~\ref{lem:PF-irred}
together show that intuitively we can consider $\scrF$ to be an
extension of $\scrL$ which ``adds the missing join-irreducible elements of
$\scrL$'' which are needed to overcome the obstacle that some lattices
do not have enough join-irreducible elements to construct a representation.

\begin{lemma} \cite{Nat2017a}*{Th.~8.8 Sublem.~A} \label{lem:ji-separates-all}
If $x \in \scrL$ and $\PF{x} \not\geq f$ in $\scrF$,
then there exists a $p \in \scrP$ for which
$p \leq f$ but $p \not\leq \PF{x}$.
\end{lemma}
\begin{proof} \disconnect
Define $H = \{h \in \scrF \mvert h \leq f \textup{ and } h \not\leq \PF{x}\}$.
$H$ is non-empty because $f \in H$.
$H$ is a poset ordered by inverse inclusion, inherited from $\scrF$.
Using the Hausdorff maximal principle,%
\footnote{The Hausdorff maximal principal is equivalent to the axiom
of choice.}
select $C$ which is a maximal
totally ordered subset of $H$.  Since $H$ is non-empty, $C$ is
non-empty (but might not contain $f$).

Define $p = \bigcup C$.  We will show $p$ is a lattice filter.
Because $C$ is a non-empty set of non-empty
sets, $p$ is non-empty.
If $a \in p$ and $b \geq a$, there is an element $q \in C$ for which
$a \in q$.  Since $q$ is a lattice filter, $b \in q \subset p$.
If $a,b \in p$, then there are elements $q,r \in C$ for which
$a \in q$ and $b \in r$.  Since $C$ is a subset of $\scrF$ that is
totally ordered, it is totally ordered by inverse containment, so one
of $q$ and $r$ contains the other; say $q \subset r$.
Then $a$ and $b$ are both $\in r$, and since $r$ is a lattice filter,
$a \wedge b \in r \subset p$.  Together, these show $p$ is a lattice
filter and so $\in \scrF$.

$p$ contains every element of $C$, so by inverse inclusion,
$p \leq$ every element of $C$.

Because $x$ is in no member of $C$, $x \not\in \bigcup C = p$
and so, by Lem.~\ref{lem:equiv}, $p \not\leq \PF{x}$.

Since $C$ is non-empty, it contains at least one element $q$.
Thus $q \leq f$, (by inverse inclusion) $f \subset q \subset \bigcup C = p$, so
$f \subset p$ and $p \leq f$.

We prove $p$ is join-irreducible in $\scrF$ by cases:
$p$ cannot be join-reducible by being $= \zerohat_\scrF = \scrL$
because $x \not\in p$.
The other way $p$ could be join-reducible is for there to be
$g, h \in \scrF$ for which $p = g \vee h$ and $g, h < p$
Since $p \not\leq \PF{x}$, either $g \not\leq \PF{x}$ or $h \not\leq \PF{x}$.
Without loss of generality, assume $g \not\leq \PF{x}$.
Since $g < p \leq f$, $g \in H$.

But since $g < p = \bigcup C$, $\bigcup C \subsetneq g$.
Those imply that $g$ strictly contains every element of $C$ which is that
$g$ is strictly less than (under inverse inclusion)
every element of $C$, and so $g \not\in C$.
Thus $\{g\} \sqcup C$ is a totally ordered subset of
$\scrF$ that strictly contains $C$, which contradicts that $C$ is
maximal.

Thus $p$ is join-irreducible and by Lem.~\ref{lem:ji-prime},
$p$ is prime.
\end{proof}

\begin{corollary} \cite{Nat2017a}*{Th.~8.4} \label{cor:plf-separates}
Given $x \not\leq y$ in $\scrL$, there exists a
$p \in \scrP$ for which $x \in p$ and $y \not\in p$.
\end{corollary}
\begin{proof} \disconnect
Define $f = \PF{x}$.  Then $\PF{y} \not\geq \PF{x} = f$.  Apply
Lem.~\ref{lem:ji-separates-all} to $y$ and $f$ to obtain $p \in \scrP$
with $p \leq f = \PF{x}$ and $p \not\leq \PF{y}$.
By Lem.~\ref{lem:equiv}, those are equivalent to $x \in p$ and
$y \not\in p$.
\end{proof}

\begin{lemma} \label{lem:phi}
Given any $x \in \scrL$, $I = \{p \in \scrP \mvert x \in p\}$ is an order
ideal of $\scrP$.
\end{lemma}
\begin{proof} \disconnect
To show that $I$ is an order ideal, assume that $p, q \in \scrP$,
$p \in I$, and $q \leq p$.  Then $x \in p$ and $p \subset q$, so
$x \in q$ and $q \in I$.
\end{proof}

\begin{definition} \cite{Nat2017a}*{Th.~8.5} \label{def:phi}
We define $\phi: \scrL \rightarrow \scrI :
x \mapsto \{p \in \scrP \mvert x \in p\}$.
\end{definition}

Lem.~\ref{lem:equiv} can be extended:

\begin{lemma} \label{lem:equiv-phi}
Given $x \in \scrL$ and $p \in \scrP$, the following are equivalent:
\begin{enumerate}
\item $p \leq \PF{x}$,
\item $\PF{x} \subset p$,
\item $x \in p$, and
\item $p \in \phi(x)$.
\end{enumerate}
\end{lemma}

\begin{theorem}  \cite{DavPriest1990a}*{sec.~10.18} \cite{Nat2017a}*{Th.~8.5} \label{th:phi-embed}
$\phi$ is a lattice embedding of $\scrL$ into $\scrI$.
\end{theorem}
\begin{proof} \disconnect
Consider $x,y \in \scrL$ and $p \in \scrP$.
Then $x \wedge y \in p$ is equivalent to ``$x \in p$ and $y \in p$''
because ($\Rightarrow$) lattice filters are closed upward and
($\Leftarrow$) lattice filters are meet-closed.
Thus $\phi(x \wedge y) = \phi(x) \cap \phi(y)$.

But also $x \vee y \in p$ is equivalent to ``$x \in p$ or $y \in p$''
because ($\Rightarrow$) $p$ is a prime filter and
($\Leftarrow$) lattice filters are closed upward.
Thus $\phi(x \vee y) = \phi(x) \cup \phi(y)$.

$\phi$ is injective because if $\phi(x) = \phi(y)$, the contrapositive
of Cor.~\ref{cor:plf-separates} shows that $x \leq y$, and similarly
$y \leq x$, and thus $x = y$.
\end{proof}

\begin{corollary} \label{cor:ring-sets}
Every distributive lattice can be represented as a ``ring of
sets'' under intersection and union.
\end{corollary}

The overall scheme of the representation is shown in
Fig.~\ref{fig:LPI}.

\begin{figure}[ht]
\centering
\includegraphics[page=\ipeFigLPI]{locally-finite-figs.pdf}%
\caption{The overall scheme of representation for a distributive lattice:
  $\scrL$ is the given lattice,
  $\scrP$ is the poset of prime filters of $\scrL$ (ordered by
  inverse inclusion),
  $\scrI$ is the lattice of order ideals of $\scrP$ (ordered by inclusion),
  $x \in \scrL$,
  $p_1, p_2 \in \scrP$ are two prime filters of $\scrL$ containing $x$,
  $\phi(x) \in \scrI$ is an ideal of $\scrP$,
  $p_1, p_2 \in \phi(x)$, $\phi(x)$ represents $x$ in the isomorphic
  image $\im \phi$ of $\scrL$ in $\scrI$.}
\label{fig:LPI}
\end{figure}

\begin{definition} \label{def:sur-ji}
Let $p$ be an element of $\scrP$ (which thus is a prime filter in $\scrF$).
We define $p$ to be a \emph{principal prime} iff
$p = \PF{x}$ for some $x \in \scrL$ (which would have to be
join-irreducible).
We define $p$ to be a \emph{secondary prime} iff
there is no such $x$.
Both terms are relative to $\scrL$, to $\scrF$, or to $\scrP$ (by abuse of
language).
\end{definition}

For $\scrL = \mbbZ \times \mbbZ$, Fig.~\ref{fig:ZxZfilters} shows that
\emph{all} prime filters are secondary primes.
Conversely, there can be elements of $\scrF$ that are not of the form
$\PF{x}$ and are not secondary primes:
$\mbbZ \times \mbbZ$ is in $\scrF$, is not an $\PF{x}$, but
also is not prime.

\section{Application to finitary lattices} \label{sec:finitary}

We first examine the specialization of the general representation
theory to finitary distributive lattices. In this setting every
lattice filter is principal, so every prime filter is the
principal filter of a join-irreducible element, and
the general construction collapses to
Birkhoff's familiar representation by ideals of join-irreducible elements,
but limited to \emph{finite} ideals.

\begin{theorem} \label{th:finitary}
If $\scrL$ is finitary then every element $p \in \scrF$ is $= \PF{x}$
for some $x \in \scrL$.
\end{theorem}
\begin{proof} \disconnect
Let $p \in \scrF$.
Since $p$ is non-empty, there is an element $y \in p$.
Since $\scrL$ is finitary, $\PI{y}$ is finite, and therefore
$\PI{y} \cap p$ is finite and non-empty.
$\PI{y} \cap p$ must contain a minimal element $x$ and necessarily
$x \leq y$.
Since $p$ is closed under meets, every element $z \in p$ must be
$\geq x$, as otherwise $z \wedge x$ would be be $< x \leq y$ and $\in p$,
contradicting that $x$ is minimal in $\PI{y} \cap p$.
Thus $x$ is minimum in $p$.  Since $p$ is a lattice filter, $p = \PF{x}$.
\end{proof}

This theorem yields several simplifications of the general representation.

\begin{corollary} \cite{CrawDil1973a}*{ch.\ 5} \cite{Stan2012a}*{Prop.~3.4.3} \label{cor:finitary}
If $\scrL$ is distributive and finitary then:
(1) $\scrP$ contains no secondary primes, so $\scrP$ is (isomorphic
to) the poset of join-irreducible elements of $\scrL$.
(2) Every value of $\phi$ is an element of $\scrI$ that is a finite
order ideal of $\scrP$.
(3) For every finite $I \in \scrI$,
$I = \{\PF{x_1},\PF{x_2}, \PF{x_3}, \ldots, \PF{x_n}\}$
for some finite series $x_\bullet$ in $\scrL$, and $\phi(\bigvee_i x_i) = I$.%
\footnote{Note that $n=0$ is allowed.  In that case,
$\bigvee_i x_i = \bigvee \myemptyset = \zerohat_\scrL$ and
$\phi(\zerohat_\scrL) = \myemptyset$.}
(4) $\phi$ is a lattice isomorphism between $\scrL$ and the sublattice
of elements of $\scrI$ that are finite order ideals of $\scrP$.
\end{corollary}

Statement (4) recovers the usual representation of a finitary
distributive lattice as the lattice of finite order ideals of its
poset of join-irreducible elements.

Thus, in the finitary case, the prime-filter representation requires
no objects beyond elements of $\scrL$; every prime filter is
determined by a lattice element, and the representation reduces to the
classical join-irreducible framework.

\section{Foundations for locally-finite lattices} \label{sec:tools}

This section gives general definitions and lemmas that will be
used deriving the representation theorem for locally-finite
lattices.  The reader can skip this section on the first reading,
referring back to it as needed.

There are at least eight different definitions of ``locally-finite''
in use, so we clarify:

\begin{definition} \label{def:locally-finite}
We define a poset or lattice to be \emph{locally-finite} iff every
interval between two elements is finite (when considered as a set of
elements).
\end{definition}

\begin{lemma} \label{lem:cover-one}
Given a locally-finite distributive lattice $\scrL$,
an element $x \in \scrL$ is join-irreducible
iff $x$ covers exactly one element of $\scrL$.
Dually, $x \in \scrL$ is meet-irreducible
iff $x$ is covered by exactly one element of $\scrL$.
\end{lemma}
\begin{proof} \disconnect
Regarding $\Rightarrow$ for join-irreducibility:
Since $x$ is join-irreducible, by Def.~\ref{def:irred}, it cannot be
$\zerohat_\scrL$ if that exists, so there is an element $y < x$.
Since $\scrL$ is
locally-finite, there must be an element $z$ such that
$y \leq z \lessdot x$.
If $x$ covered more than one element of $\scrL$, then $x$ would be the
join of any two of those elements and would not be join-irreducible.
Thus $x$ covers exactly one element.

Regarding $\Leftarrow$ for join-irreducibility:
Suppose $x$ covers exactly one element $y$.
There are two possible ways $x$ could be join-reducible.
One is if $x = \zerohat_\scrL$.  But that is impossible, since $x$
covers an element.
The other possible way that $x$ could be join-reducible is if there exists
$z, w < x$ with $z \vee w = x$.
But then, since $\scrL$ is locally finite and $y$ is the only element
covered by $x$, $z, w \leq y$,
implying $z \vee w \leq y < x$, which is a contradiction.
Thus $x$ is not join-reducible for either reason and so it must be
join-irreducible.

Regarding meet-irreducibility:  This is proved dually.
\end{proof}

\begin{lemma} \label{lem:diff-one}
Given a distributive lattice $\scrL$ and the function $\phi$ defined in
Def.~\ref{def:phi},
if $x \lessdot y$ in $\scrL$, then $\phi(x) \subset \phi(y)$ and
$\phi(y) \setminus \phi(x)$ contains exactly one element.
\end{lemma}
\begin{proof} \disconnect
By Def.~\ref{def:phi}, it is clear that $\phi(x) \subset \phi(y)$.

By Cor.~\ref{cor:plf-separates},
there exists a $p \in \scrP$ for
which $y \in p$ and $x \not\in p$.
Thus $p \in \phi(y) \setminus \phi(x)$.

Assume that $p \neq q$ are both $\in \phi(y) \setminus \phi(x)$,
so $y \in p,q$ and $x \not\in p,q$.

Consider any $z \in p$.
Then $z \wedge y \in p$, $z \wedge y \leq y$,
and $z \wedge y \not\leq x$.
So $(z \wedge y) \vee x > x$ and since
$x \lessdot y$, $(z \wedge y) \vee x = y$.
Then $(z \wedge y) \vee x \in q$, which, since $q$ is prime, implies that
either $z \wedge y \in q$ or $x \in q$.
But since $x \not\in q$, $z \wedge y \in q$, implying $z \in q$.
Thus $p \subset q$.

Similarly $q \subset p$, and so $p = q$.
Thus $\phi(y) \setminus \phi(x)$ has exactly one element.
\end{proof}

\begin{definition} \label{def:graded}
We define a locally-finite poset or lattice to be \emph{locally
ranked} if for any $x \leq y$,
all maximal chains between $x$ and $y$ have the same length.%
\footnote{A locally-ranked poset or lattice is also said to satisfy
the Jordan--Dedekind chain condition.}
In such a poset or lattice, for any $x \leq y$, the \emph{rank difference}
$\rho(x,y)$ is defined to be the length of any of the maximal chains
between $x$ and $y$.
\end{definition}

\begin{lemma} \label{lem:diff-n}
Given a locally-finite distributive lattice $\scrL$
and the function $\phi$ defined in Def.~\ref{def:phi},
$\scrL$ is locally ranked, and
for any $x \leq y$ in $\scrL$, $\phi(x) \subset \phi(y)$ and
$|\phi(y) \setminus \phi(x)| = \rho(x, y)$, where
$\rho(x, y)$ is the rank difference of $y$ above $x$.
\end{lemma}
\begin{proof} \disconnect
Consider any $x \leq y$ in $\scrL$. Choose any maximal chain between them,
$x = z_0 \lessdot z_1 \lessdot z_2 \lessdot \cdots \lessdot z_n = y$.
Applying Lem.~\ref{lem:diff-one}, for any $0 \leq i < n$,
$\phi(z_i) \subset \phi(z_{i+1})$ and
$|\phi(z_{i+1}) \setminus \phi(z_i)| = 1$.
Thus $\phi(x) \subset \phi(y)$ and $|\phi(y) \setminus \phi(x)| = n$.
That shows that all maximal chains between $x$ and $y$ have the same
length, which is $|\phi(y) \setminus \phi(x)|$, and thus $\scrL$ is
locally ranked
with $\rho(x,y) = |\phi(y) \setminus \phi(x)|$.
\end{proof}

\begin{definition} \label{def:symdiff}
We define $\Symdiff$ to be the symmetric difference of two sets.
We define that $A \sim B$ iff $A \Symdiff B$ is finite.
Since $\sim$ is an equivalence relation, we define a \emph{\simclass}
to be an equivalence class of $\sim$ (within whatever class of sets is under
consideration).
\end{definition}

\begin{lemma} \label{lem:symdiff-lf}
In any poset of sets ordered by inclusion, any
\simclass is locally-finite as a subposet of the poset.
Similarly, if the poset of sets is ordered by inverse inclusion, any
\simclass is locally-finite as a subposet of the poset.
\end{lemma}
\begin{proof} \disconnect
Regarding a poset of sets ordered by inclusion:

Given any two sets $A$ and $B$ in the poset of sets,
with $A \sim B$ and $A \leq B$, we must show that the interval
between them contains only a finite number of elements.

Consider all sets $C$ in the poset with $A \subset C \subset B$.
Then $A \Symdiff C \subset A \Symdiff B$.
Since $A \Symdiff B$ is finite, there are at most
$2^{|A \Symdiff B|} = n$ distinct values of $A \Symdiff C$
generated by all of the $C$, and since each distinct $C$ generates a
distinct value of $A \Symdiff C$, there are at most a finite number $n$ of
distinct possible $C$.

Thus the interval between $A$ and $B$ in the poset contains a finite
number of elements, showing the \simclass is locally finite.

Regarding a poset of sets ordered by inverse inclusion:
The proof is similar.
\end{proof}

\begin{lemma} \label{lem:step-L-to-I}
Given a distributive lattice $\scrL$,
the function $\phi$ defined in Def.~\ref{def:phi},
and
a sequence of elements $x_0$, $x_1$, $x_2$, \ldots, $x_n$ in $\scrL$
where for each
$0 \leq i < n$, either $x_i \lessdot x_{i+1}$ or $x_i \gtrdot x_{i+1}$,
then for each $i$, $|\phi(x_i) \Symdiff \phi(x_{i+1})| = 1$
and $\phi(x_0) \sim \phi(x_i)$.
\end{lemma}
\begin{proof} \disconnect
This is proved by application of Lem.~\ref{lem:diff-one} to each
adjacent pair $x_i$, $x_{i+1}$ and the transitivity of $\sim$.
\end{proof}

\begin{definition}
When we use the set difference notation $X \setminus Y$,
we imply/require that $Y \subset X$.
We use the notation $X \sqcup Y$ to denote the disjoint union of $X$
and $Y$, that is, it implies/requires that $X \cap Y = \myemptyset$.
\end{definition}

We use a concept of when an
element of a poset $\scrP$ is ``finitely far'' from
a particular ideal.  Specifically, an element is ``reachable'' from
the ideal, as shown in Fig.~\ref{fig:reachable},
when only a finite number of elements need to be added to or
deleted from the ideal to construct an ideal
that either includes or excludes the element.

\begin{figure}[ht]
\centering
\begin{minipage}{0.45\textwidth}
  \centering
  \includegraphics[page=\ipeFigReachable]{locally-finite-figs.pdf}%
  \caption{The elements $p$ and $q$ are reachable from the order ideal
    $I$ because the shaded regions next to each contain a finite number
    of elements of $\scrP$.}
  \label{fig:reachable}
\end{minipage}
\begin{minipage}{0.45\textwidth}
  \centering
  \includegraphics[page=\ipeFigReachMessy]{locally-finite-figs.pdf}%
  \caption{Configuration for Lem.~\forwardref{lem:reach-messy}:
    The shaded area is the zone of the upper boundaries of the members
    of $\scrD$, a \simclass of order ideals.
    $Q$, $P$, and $q$ are given.
    Then $R$ exists and resembles $P$ except that it is in $\scrD$.}
  \label{fig:reach-messy}
\end{minipage}
\end{figure}

\begin{definition} \label{def:reach}
We define that an element $p$ of a poset $\scrP$ is
\emph{reachable}
from an order ideal $I \subset \scrP$ iff
there exist order ideals $I^+, I^- \subset \scrP$ with
$p \in I^+$; $p \not\in I^-$; and $I^+, I^- \sim I$.
\end{definition}

\begin{lemma} \label{lem:reach}
Given an element $p$ of a poset $\scrP$ and an order ideal
$I \subset \scrP$, then $p$ is reachable from $I$ iff
$\PI{p} \setminus (I \cap \PI{p})$ and $\PF{p} \cap I$ are finite sets.
\end{lemma}
\begin{proof} \disconnect
Regarding $\Rightarrow$:
If $p$ is reachable, choose $I^+ \sim I$ containing $p$. Since $I^+$
is an ideal,
$\PI{p} \subset I^+$, and then
$\PI{p} \setminus (I \cap \PI{p}) =
\PI{p} \setminus (I \cap I^+ \cap \PI{p}) =
\PI{p} \setminus (I \cap I^+) \subset
I^+ \setminus (I \cap I^+) \subset
I^+ \Symdiff I$.
Since the last of those is finite, $\PI{p} \setminus (I \cap \PI{p})$
is finite.
Similarly, choose $I^- \sim I$ excluding $p$. No element $\geq p$ can
lie in $I^-$ since an ideal containing such an element would contain
$p$.
Hence $\PF{p} \cap I \subset
I \setminus (I^- \cap I) \subset I^- \Symdiff I$.
Since the last of those is finite, $\PF{p} \cap I$ is finite.

Regarding $\Leftarrow$:
If
$\PI{p} \setminus (I \cap \PI{p})$ and $\PF{p} \cap I$ are finite, define
$I^+ = I \cup \PI{p}$ and $I^- = I \setminus (\PF{p} \cap I)$.
Both are order ideals, $I^+$ contains $p$, and $I^-$ excludes $p$.
$I^+ \Symdiff I = (I \cup \PI{p}) \Symdiff I = \PI{p} \setminus (I \cap \PI{p})$
and
$I^- \Symdiff I = (I \setminus (\PF{p} \cap I)) \Symdiff I =
\PF{p} \setminus (I \cap \PF{p})$ are assumed to be finite, so
$I^+ \sim I$ and $I^- \sim I$.  Thus $p$ is reachable.
\end{proof}

\begin{lemma} \label{lem:reach-equiv}
Given two order ideals $I$ and $J$ of $\scrP$ with $I \sim J$, and any
$p \in \scrP$, then $p$ is reachable from $I$ iff $p$ is reachable
from $J$.
\end{lemma}
\begin{proof} \disconnect
Given that $p$ is reachable from $I$, there are $I^\pm$ that
include/exclude $p$ and are $\sim I$.  Then by transitivity $I^\pm \sim J$,
showing $p$ is reachable from $J$.
Symmetrically, $p$ is reachable from $J$ implies $p$ is reachable from
$I$.
\end{proof}

Thus reachability depends only on the \simclass of the ideal, and so:

\begin{definition}
We define an element to be \emph{reachable} from a \simclass of ideals
iff it is reachable from any ideal in the \simclass.
\end{definition}

\begin{lemma} \label{lem:reach-convex}
  Given a poset $\scrP$ and an order ideal $I \subset \scrP$,
  the subset of $\scrP$ of elements reachable from $I$ are a convex
  subposet of $\scrP$.
\end{lemma}
\begin{proof} \disconnect
  Suppose $x \leq z \leq y$ in $\scrP$, and $x$ and $y$ are reachable
  from $I$.  There exists an order ideal $X^-$ with $x \not\in X^-$
  and $X^- \sim I$.  Similarly, there exists an order ideal $Y^+$ with
  $y \in Y^+$ and $Y^+ \sim I$.  Then since $Y^+$ is an order ideal,
  $z \in Y^+$ and since $X^-$ is an order ideal, $z \not\in X^-$.
  Thus, $z$ is reachable from $I$.
\end{proof}

\begin{definition} \label{def:fin-diff}
Given an order ideal $I$ of a poset $\scrP$, we define
$$ \scrD_I = \{ Q \in \scrI \mvert Q \sim I \},$$
the \simclass of ideals containing $I$.
\end{definition}

\begin{lemma} \label{lem:DP-sublattice}
Given an order ideal $I$ of a poset $\scrP$, $\scrD_I$ is a convex
sublattice of the lattice of order ideals of $\scrP$.
\end{lemma}
\begin{proof} \disconnect
Let $J$ and $K$ be any members of $\scrD_I$.
Then
$(J \cup K) \Symdiff I \subset (J \Symdiff I) \cup (K \Symdiff I)$
and $(J \cap K) \Symdiff I \subset (J \Symdiff I) \cup (K \Symdiff I)$.
Since the right-hand sides of both of those are finite, the left-hand
sides are also, and so $J \cup K$ and $J \cap K$ are in $\scrD_I$.
Thus $\scrD_I$ is a sublattice, closed under union and intersection.

Assume that $J \subset K$ in $\scrD_I$, and
let $L$ be any order ideal of $\scrP$ with $J \subset L \subset K$.
Then
$L \Symdiff I \subset (J \Symdiff I) \cup (K \Symdiff I)$,
showing $L \Symdiff I$ is finite.
Thus $L \in \scrD_I$, showing $\scrD_I$ is convex,
closed under betweenness.
\end{proof}

Def.~\ref{def:reach} means that reachability implies that ideals with
certain properties exist.  This can be extended to additional useful
properties.

\begin{lemma} \label{lem:reach-messy}
Given a poset $\scrP$,
a \simclass $\scrD$ of order ideals of $\scrP$,
an order ideal $P \subset \scrP$,
another order ideal
$Q \subset \scrP$ with $P \subset Q$ and $Q \in \scrD$,
and an element $q \in \scrP$ that is reachable from $\scrD$ with
$q \in Q \setminus P$,
then there exists an order ideal $R \in \scrD$ with
$P \subset R \subsetneq Q$ and $q \not\in R$.
\end{lemma}
\begin{proof} \disconnect
Fig.~\ref{fig:reach-messy} shows the configuration for this lemma.

Since $q$ is reachable from $\scrD$, Def.~\ref{def:reach} shows there
exists an ideal $S \in \scrD$ with $q \not\in S$.
Construct $R = (S \cap Q) \cup P$.
Clearly $P \subset R$.
Since $S \cap Q \subset Q$ and $P \subset Q$, $R \subset Q$.
Since $q \not\in P$ and $q \not\in S$, $q \not\in R$
and so $R \neq Q$.

We now show $R \in \scrD$:
Since $S$ and $Q$ are $\in \scrD$,
Lem.~\ref{lem:DP-sublattice} shows $S \cap Q \in \scrD$.
Since $P \subset Q$,
$S \cap Q \subset (S \cap Q) \cup P \subset Q$
and since $S \cap Q$ and $Q$ are $\in \scrD$,
Lem.~\ref{lem:DP-sublattice} shows $R = (S \cap Q) \cup P \in \scrD$.

Thus an $R$ with the required properties exists.
\end{proof}

\begin{lemma} \label{lem:reach-include-exclude}
Given a poset $\scrP$,
a \simclass $\scrD$ of order ideals of $\scrP$, and
given $p$ and $q$ reachable from $\scrD$ with $p \not\leq q$,
then there exists an ideal $R \in \scrD$ with
$p \not\in R$ and $q \in R$.
\end{lemma}
\begin{proof} \disconnect
Because $p \not\leq q$, $\PF{p}$ and $\PI{q}$ are disjoint.
Choose any ideal in $\scrD$ as $S$ and construct
the ideals $T = S \setminus (\PF{p} \cap S)$ and $R = T \cup \PI{q}$.
By Lem.~\ref{lem:reach}, $\PF{p} \cap S$ is finite and so $T \sim S$.
Similarly, $\PI{q} \setminus (T \cap \PI{q})$ is finite and so $R \sim T$.
Thus $R \in \scrD$ and clearly $p \not\in R$ and $q \in R$.
Thus an $R$ with the required properties exists.
\end{proof}

\section{Application to locally-finite lattices} \label{sec:locally-finite}

We now specialize the general representation theory to locally-finite
distributive lattices.  This is valuable because in combinatorics many
interesting lattices aren't finitary but are locally finite, and
local-finiteness of $\scrL$ imposes combinatorial structure on $\scrI$
which allows us to identify which elements of $\scrI$ form
the isomorphic image of $\scrL$.
Unlike the finitary case, secondary prime filters may occur, so
the general form of $\scrP$ introduced in sec.~\ref{sec:general} becomes
essential.

The representation theorem can be extended to be a duality between
locally-finite distributive lattices and
the \emph{pairs} defined below.
This parallels Birkhoff's duality between finite distributive lattices
and finite posets. Specifically,

\begin{definition} \label{def:pair}
In this paper, a \emph{pair} is a pair of
\begin{enumerate}
\item a poset and
\item a distinguished \simclass (Def.~\ref{def:symdiff} of order ideals
of the poset
\end{enumerate}
for which all elements of the poset are reachable from the \simclass
(Def.~\ref{def:reach} and Lem.~\ref{lem:reach-equiv}).
\end{definition}

\begin{definition} \label{def:forward}
The representation of the duality is:
From the locally-finite lattice $\scrL$ we construct its poset of
prime filters $\scrP$ (ordered by inverse inclusion).
From $\scrP$ we construct its lattice of order
ideals $\scrI$ (ordered by inclusion).
The representation isomorphism (Def.~\ref{def:phi})
$\phi: x \mapsto \{p \in \scrP \mvert x \in p\} $ maps $\scrL$ onto
$\im \phi$, which is necessarily a \simclass $\scrD$ of $\scrI$
(Lem.~\forwardref{lem:lf-repr}).
The resulting pair is $(\scrP, \scrD)$.
\end{definition}

The representation of this duality is shown in Fig.~\ref{fig:LPIlf}:
Note that the figure shows the duality rather awkwardly; the
locally-finite lattice $\scrL$ is dual to the pair of the poset $\scrP$
and \simclass $\scrD$ within $\scrI$ ($\scrI$ being determined by $\scrP$).
\begin{figure}[ht]
\centering
\includegraphics[page=\ipeFigLPIlf]{locally-finite-figs.pdf}%
\caption{The overall scheme of the representation for a
  locally-finite distributive lattice:
  $\scrL$ is the given lattice,
  $\scrP$ is the poset of prime filters of $\scrL$,
  $\scrI$ is the lattice of order ideals of $\scrP$,
  $x \in \scrL$,
  $p_1, p_2 \in \scrP$ are two example prime filters of $\scrL$ containing $x$,
  $\phi(x) \in \scrI$ is an ideal of $\scrP$,
  $p_1, p_2 \in \phi(x)$, $\phi(x)$ represents $x$ in the isomorphic
  image $\scrD$ of $\scrL$ in $\scrI$,
  $D$ is an example element of $\scrD$,
  $\scrD = \scrD_D = \scrD_{\phi(x)} = \im \phi$ is a \simclass and
  connected component of $\scrI$.}
\label{fig:LPIlf}
\end{figure}

\begin{definition} \label{def:reverse}
The reconstruction of the duality is considerably simpler:
From the pair of a poset $\scrP$ and its
\simclass $\scrD$ within $\scrI$ (where $\scrI$ is determined by $\scrP$),
we extract $\scrD$ as a lattice, which is necessarily locally finite
(Lem.~\ref{lem:symdiff-lf}).
\end{definition}

The subtle aspect of this construction is that given a locally-finite lattice
$\scrD$, there is exactly one poset
$\scrP$ (up to isomorphism) that it can be paired with.
(Although a particular $\scrP$ may be
pairable with more than one of its \simclasses, and some of those may
be isomorphic to each other.)

\begin{theorem} \label{th:duality}
The representation and reconstruction form a duality,
that is, they are quasi-inverses.
\end{theorem}
\begin{proof}
  The representation and reconstruction form a duality because,
  as will be shown below:
\begin{enumerate}
\item The representation produces a pair of a poset and
  \simclass of its order ideals for which all elements of the poset
  are reachable. (Th.~\forwardref{th:forward})
\item The reconstruction produces a locally-finite lattice.
  (Th.~\forwardref{th:converse-repr})
\item Both constructions are isomorphism-preserving, in that if they start with
  isomorphic objects, they result in isomorphic objects.
  (Th.~\forwardref{th:isom-preserving})
\item The representation followed by the reconstruction
  produces a lattice isomorphic to the starting lattice.
  (Th.~\forwardref{th:forward-reverse})
\item The reconstruction followed by the representation
  produces a pair isomorphic (Def.~\forwardref{def:pair-isom}) to the
  starting pair.  (Th.~\forwardref{th:reverse-forward})
\end{enumerate}
\end{proof}

To support this, we must define isomorphism of the representing
pairs:

\begin{definition} \label{def:pair-isom}
We define that
two pairs $(\scrP, \scrD)$ and $(\scrP^\prime, \scrD^\prime)$ of
posets with respective \simclasses of order ideals are \emph{isomorphic}
iff there exists a poset isomorphism $\psi: \scrP \rightarrow \scrP^\prime$
which substitutes the elements-of-elements of $\scrD$ to make
$\scrD^\prime$, that is,
$\psi_{\#\#}(\scrD) = \scrD^\prime$
where $\psi_{\#\#}$ is the image map of the image map\footnote{%
We use ``image map'' rather than ``induced map on power sets''.}
of $\psi$.
\end{definition}

\begin{theorem} \label{th:isom-preserving}
It is straightforward that the two constructions take isomorphic
objects into isomorphic objects because they are composed of operations
with that property.
\end{theorem}

Since lattices with $\zerohat$ are often of specific concern in
combinatorics, we note the following restriction of the duality, which
extends the result for finitary lattices, Th.~\ref{th:stanley}.

\begin{theorem} \label{th:zerohat}
  If the lattice $\scrL$ and the pair $(\scrP, \scrD)$ are dual, then
  $\scrL$ has a minimum element $\zerohat$ iff
  $\scrD$ is the \simclass of finite order ideals of $\scrP$.
\end{theorem}
\begin{proof} \disconnect
Regarding $\Rightarrow$:  Given $\zerohat \in \scrL$, by
Def.~\forwardref{def:D}, $\scrD = \scrD_{\phi(\zerohat)}$.
By Def.~\ref{def:phi}, $\phi(\zerohat) = \{p \in \scrP \mvert \zerohat \in p\}$.
However, if $\zerohat \in p \in \scrP$ for a prime filter $p$ of $\scrL$,
then $p = \scrL$, in which case
$p$ is not a proper filter of $\scrL$, which is a contradiction.
Thus $\phi(\zerohat) = \myemptyset$, and by Def.~\ref{def:fin-diff},
$\scrD$ is the set of finite order ideals of $\scrP$.

Regarding $\Leftarrow$:
Given that $\scrD$ is the finite order ideals of $\scrP$,
$\myemptyset \in \scrD$.  Then necessarily, $\myemptyset$ is the minimum
element of $\scrD$, and since $\scrL$ is isomorphic to $\scrD$,
$\scrL$ has a minimum element.
\end{proof}

\subsection{The representation}

In this section we prove that the result of the representation
is a pair with the required properties, particularly that the
constructed $\scrD$ is a \simclass and that every element of $\scrP$
is reachable from $\scrD$.
The course of the argument is as follows:
First we elucidate the local structures of $\scrL$ and $\scrI$.
Next we introduce the ``raising'' operation, which constructs the
correspondence between (1) moving from an element $x$ of $\scrL$ to one
of its covers $y$ and (2) adding an element of $\scrP$ to the ideal
$\phi(x)$ to make $\phi(y)$.
Similarly we introduce the ``lowering'' operation, which constructs the
correspondence between (1) moving from an element $x$ of $\scrL$ to an
element $y$ it covers and (2) removing an element of $\scrP$ from the ideal
$\phi(x)$ to make $\phi(y)$.
These allow us to show that the image of $\phi$ is exactly the
\simclass of $\phi(x_0)$ for any $x_0 \in \scrL$, the set of all
ideals in $\scrI$ whose symmetric difference from $\phi(x_0)$ is
finite.

Considering the Hasse diagram of $\scrI$ as a graph, this also shows that
the image of $\phi$ is exactly one connected component of the graph.
(See sec.~\ref{sec:graph}.)

\begin{lemma} \label{lem:raising}
Given $x \in \scrL$ and $p \in \scrP$ where $x \not\in p$,
there exists a $y \in \scrL$ for which  $y > x$, $y \in p$, and
$y$ is the minimum element of $\PF{x} \vee p = \PF{x} \cap p$
and thus $\PF{y} = \PF{x} \vee p = \PF{x} \cap p$.
\end{lemma}
\begin{proof} \disconnect
Define $f = \PF{x} \vee p = \PF{x} \cap p$.
Since $f$ is a lattice filter, by Def.~\ref{def:lfilter}, it is not
empty and we can choose
$z \in f$.  Necessarily $z > x$ and since $\scrL$ is locally
finite, the interval $[x,z]$ is finite and non-empty.
Thus $f \cap [x,z]$ is finite and non-empty.
Choose $y$ to be minimal in $f \cap [x,z]$.  This implies $y$ is
minimal in $f$ because any smaller element of $f$ would also be an
element of $f \cap [x,z]$.

Any minimal $y$ must be the minimum of $f$, because if $y_1 \neq y_2$ are both
minimal in $f$, since $f$ is a lattice filter, $y_1 \wedge y_2$ is also in
$f$, contradicting that $y_1$ and $y_2$ are both minimal in $f$.
Thus, $f = \PF{y}$.

Clearly, $y \geq x$.  Because $y \in p$ and $x \not\in p$, $y \neq x$.
\end{proof}

\begin{definition} \label{def:raising}
Given $x \in \scrL$ and $p \in \scrP$ where $x \not\in p$,
we define $x \Raise p$ to be the minimum element of
$\PF{x} \vee p = \PF{x} \cap p$.
\end{definition}

\begin{lemma} \label{lem:add-one}
Suppose $x \in \scrL$ and $p \in \scrP$ such that
$p \not\in \phi(x)$ and
$\phi(x) \sqcup \{p\}$ is an order ideal of $\scrP$.
Then $x \lessdot x \Raise p$ and
$\phi(x \Raise p) = \phi(x) \sqcup \{p\}$.
\end{lemma}
\begin{proof} \disconnect
Define $y = x \Raise p$.

We will show that $\phi(y) = \phi(x) \sqcup \{p\}$.
Regarding $\phi(x) \sqcup \{p\} \subset \phi(y)$:
Since $x < y$, $\phi(x) \subset \phi(y)$.
Since $y \in p$, $p \in \phi(y)$.
Thus $\phi(x) \sqcup \{p\} \subset \phi(y)$.

Regarding $\phi(y) \subset \phi(x) \sqcup \{p\}$:
For any $q \in \phi(y)$,
$q \leq \PF{y} = \PF{x} \vee p$.
Since $\scrF$ is distributive and $q$ is join-irreducible in $\scrF$
by Lem.~\ref{lem:ji-prime},
$q \leq \PF{x}$ or $q \leq p$.
That is equivalent to ``$q \in \phi(x)$ or $q \leq p$''.
Since $\phi(x) \sqcup \{p\}$ is an order ideal of $\scrP$ by hypothesis,
either case implies $q \in \phi(x) \sqcup \{p\}$.
Thus $\phi(y) \subset \phi(x) \sqcup \{p\}$.

Together, these show $\phi(y) = \phi(x) \sqcup \{p\}$.
By Lem.~\ref{lem:raising}, $x < y$, and since $\phi(y)$ has exactly one more
element than $\phi(x)$, Lem.~\ref{lem:diff-n} shows that $x \lessdot y$.
\end{proof}

Dually, regarding lattice ideals:

\begin{lemma} \label{lem:lowering}
Given $x \in \scrL$ and $p^\prime \in \scrP^\prime$ where $x \not\in p^\prime$,
there exists a $y \in \scrL$ for which $y < x$, $y \in p^\prime$, and
$y$ is the maximum element of $\PI{x} \wedge p^\prime = \PI{x} \cap p^\prime$
and thus $\PI{y} = \PI{x} \wedge p^\prime = \PI{x} \cap p^\prime$.
\end{lemma}
\begin{proof} \disconnect
Define $f^\prime = \PI{x} \wedge p^\prime = \PI{x} \cap p^\prime$.
Since $f^\prime$ is a lattice ideal, by Def.~\ref{def:lfilter}, it is
not empty, and
we can choose $z \in f^\prime$.
Necessarily $z < x$ and since $\scrL$ is locally finite,
the interval $[z,x]$ is finite and non-empty.
Thus $f^\prime \cap [z,x]$ is finite and non-empty.
Choose $y$ to be maximal in  $f^\prime \cap [z,x]$.
This implies $y$ is maximal in $f^\prime$ because any greater element
of $f^\prime$
would also be an element of $f^\prime \cap [z,x]$.

Any maximal $y$ must be the maximum of $f^\prime$,
because if $y_1 \neq y_2$ are both
maximal in $f^\prime$, since $f^\prime$ is a lattice ideal, $y_1 \vee
y_2$ is also in
$f^\prime$, contradicting that $y_1$ and $y_2$ are both maximal in $f^\prime$.
Thus, $f^\prime = \PI{y}$.

Clearly, $y \leq x$.
Because $y \in p^\prime$ and $x \not\in p^\prime$, $y \neq x$.
\end{proof}

\begin{definition} \label{def:lowering}
Given $x \in \scrL$ and $p^\prime \in \scrP^\prime$ where $x \not\in p^\prime$,
we define $x \Lower p^\prime$ to be the maximum element of
$\PI{x} \wedge p^\prime = \PI{x} \cap p^\prime$.
\end{definition}

\begin{lemma} \label{lem:delete-one}
Suppose $x \in \scrL$ and $p \in \scrP$ such that
$p \in \phi(x)$ and
$\phi(x) \setminus \{p\}$ is an order ideal of $\scrP$.
Then defining $p^\prime = \scrL \setminus p$,
$p^\prime \in \scrP^\prime$, $x \gtrdot x \Lower p^\prime$, and
$\phi(x \Lower p^\prime) = \phi(x) \setminus \{p\}$.
\end{lemma}
\begin{proof} \disconnect
(This proof is dual to the proof of Lem.~\ref{lem:add-one} but that is
concealed by our notation, since $\phi$ is not self-dual.)

Define $p^\prime = \scrL \setminus p$ and $y = x \Lower p^\prime$.

We will show that $\phi(y) = \phi(x) \setminus \{p\}$.
Regarding $\phi(y) \subset \phi(x) \setminus \{p\}$:
Since $y < x$, $\phi(y) \subset \phi(x)$.
By construction $p \not\in \phi(y)$.
Thus $\phi(y) \subset \phi(x) \setminus \{p\}$.

Regarding $\phi(x) \setminus \{p\} \subset \phi(y)$:
We prove this by the contrapositive:
For any $q \not\in \phi(y)$, $y \not\in q$.
Define $q^\prime = \scrL \setminus q$, so $q^\prime$ is a lattice
ideal and $y \in q^\prime$.
Since $q$ is join-irreducible in $\scrF$ by Lem.~\ref{lem:ji-prime},
$q^\prime$ is meet-irreducible in $\scrF^\prime$.
Then computing in $\scrF^\prime$,
$q^\prime \geq \PI{y} = \PI{x} \wedge p^\prime$.
Since $q^\prime$ is meet-irreducible and $\scrF^\prime$ is distributive,
$q^\prime \geq \PI{x}$ or $q^\prime \geq p^\prime$.
That is equivalent to ``$x \in q^\prime$ or $q^\prime \geq p^\prime$'',
which is equivalent to ``$x \not\in q$ or $q \geq p$ (in $\scrF$)'',
which is further equivalent to ``$q \not\in \phi(x)$ or $q \geq p$''.
Since $\phi(x) \setminus \{p\}$ is an order ideal of
$\scrP$ by hypothesis,
either case implies $q \not\in \phi(x) \setminus \{p\}$.
Thus, $q \not\in \phi(y)$ implies $q \not\in \phi(x) \setminus \{p\}$,
and so $\phi(x) \setminus \{p\} \subset \phi(y)$.

By Lem.~\ref{lem:lowering}, $y < x$, and since $\phi(y)$ has exactly one less
element than $\phi(x)$, Lem.~\ref{lem:diff-n} shows that $y \lessdot x$.
\end{proof}

\begin{lemma} \label{lem:jump-L-to-I}
Given two elements $x$ and $y$ of $\scrL$,
$\phi(x) \sim \phi(y)$.
\end{lemma}
\begin{proof} \disconnect
Construct a maximal chain from $x$ to $x \wedge y$,
$x = z_0 \gtrdot z_1 \gtrdot z_2 \gtrdot \cdots \gtrdot z_n = x \wedge y$.
Then extend the $z_\bullet$ with a maximal chain from $x \wedge y$ to
$y$,
$x \wedge y = z_n \lessdot z_{n+1} \lessdot z_{n+2} \lessdot \cdots \lessdot
z_m = y$.
Applying Lem.~\ref{lem:step-L-to-I} to the sequence $z_\bullet$ shows
that for all $0 \leq i < m$, $|\phi(z_i) \Symdiff \phi(z_{i+1})| = 1$.
Then
$\phi(x) \Symdiff \phi(y) = \phi(z_0) \Symdiff \phi(z_m)
\subset \bigcup_i \phi(z_i) \Symdiff \phi(z_{i+1})$, so
$\phi(x) \Symdiff \phi(y)$ is finite.
\end{proof}

\begin{lemma} \label{lem:step-I-to-L}
Given a sequence of elements $I_0$, $I_1$, $I_2$, \ldots, $I_n$ in
$\scrI$ where
$I_0 = \phi(x_0)$ for some $x_0$ in $\scrL$ and where for each
$0 \leq i < n$, $|I_i \Symdiff I_{i+1}| = 1$,
then for each $i$, $I_{i+1} =\phi(x_{i+1})$ for some $x_{i+1}$
in $\scrL$ and either $x_i \lessdot x_{i+1}$ or
$x_i \gtrdot x_{i+1}$.
\end{lemma}
\begin{proof} \disconnect
This is proved by induction on $n$ from 0 upward.
The conclusion is trivial if $n = 0$.
Assuming it is true for $n = N$, we will show that it is true
for $n = N+1$.

Given $n = N+1$, $N = n-1$, and
by induction, $I_{n-1} = \phi(x_{n-1})$ for some $x_{n-1} \in \scrL$.
Since $|I_{n-1} \Symdiff I_n| = 1$,
either $I_n = I_{n-1} \sqcup \{p\}$ (where $p \not\in I_{n-1}$)
or $I_n = I_{n-1} \setminus \{p\}$ (where $p \in I_{n-1}$),
for some $p \in \scrP$.
In the first case, define $x_n = x_{n-1} \Raise p$ and
apply Lem.~\ref{lem:add-one} to show $I_n = \phi(x_n)$ and
$x_{n-1} \lessdot x_n$.
In the second case, define $x_n = x_{n-1} \Lower (\scrL \setminus p)$ and
apply Lem.~\ref{lem:delete-one} to show $I_n = \phi(x_n)$ and
$x_{n-1} \gtrdot x_n$.
In either case, the conclusion is true for $n = N+1$.
\end{proof}

\begin{lemma} \label{lem:ideal-add-one}
Given two order ideals $I, J$ of a poset $P$, with
$J \subset I$ and $I \setminus J$ finite, there exists (at least one)
$x \in I \setminus J$ such that $J \sqcup \{x\}$ is an order ideal
of $P$.
\end{lemma}
\begin{proof} \disconnect
Since $I \setminus J$ is a finite subposet of $P$, it contains a
minimal element $x$.  Because $x \in I$,
any element $y < x$ in $P$ is $\in I$ and either $y \in J$ or
$y \in I \setminus J$.  But since $x$ is minimal in
$I \setminus J$, the latter cannot happen.  So any such $y$ is
$\in J$.  That shows that $J \sqcup \{x\}$ is an order ideal of $P$.
\end{proof}

\begin{lemma} \label{lem:ideal-delete-one}
Given two order ideals $I, J$ of a poset $P$, with
$J \subset I$ and $I \setminus J$ finite, there exists (at least one)
$x \in I \setminus J$ such that $I \setminus \{x\}$ is an order ideal
of $P$.
\end{lemma}
\begin{proof} \disconnect
Since $I \setminus J$ is a finite subposet of $P$, it contains a
maximal element $x$.  Since $x \not\in J$, there is no element
$y \in J$ with $x \leq y$.  Thus $x$ is maximal in all of $I$.
Thus $I \setminus \{x\}$ is also an order ideal of $P$.
\end{proof}

\begin{lemma} \label{lem:jump-I-to-L}
Given two elements $I$ and $J$ of $\scrI$,
with $I = \phi(x)$ for some $x \in \scrL$ and
$I \sim J$,
then $J = \phi(y)$ for some $y \in \scrL$.
\end{lemma}
\begin{proof} \disconnect
$I \cap J$ is in $\scrI$.
Since $I \sim J$, both $I \Symdiff (I \cap J)$ and
$J \Symdiff (I \cap J)$ are finite.
Then repeatedly applying Lem.~\ref{lem:ideal-delete-one}, we can
construct a chain of elements of $\scrI$,
$I = K_0, K_1, K_2, \ldots, K_n = I \cap J$
by successively removing from $I$ maximal remaining elements of
$I \setminus (I \cap J)$.
We can extend the chain with
$I \cap J = K_n, K_{n+1}, K_{n+2}, \ldots, K_m = J$
by repeatedly applying Lem.~\ref{lem:ideal-add-one},
successively adding to $I \cap J$ minimal un-added elements of
$J \setminus (I \cap J)$.
Applying Lem.~\ref{lem:step-I-to-L} to the sequence $K_\bullet$ shows
$J = \phi(y)$ for some $y \in \scrL$.
\end{proof}

\begin{lemma} \label{lem:lf-repr}
Given any $x \in \scrL$, $\phi$ is a lattice isomorphism between
$\scrL$ and $\scrD_{\phi(x)}$.
Thus $\im \phi$ is a \simclass and
$\scrD_{\phi(x)}$ is independent of the choice of $x \in \scrL$.
\end{lemma}
\begin{proof} \disconnect
By Th.~\ref{th:phi-embed}, $\phi$ is a lattice embedding of
$\scrL$ into $\scrI$, so
what remains to be shown is that the image of $\phi$ is $\scrD_{\phi(x)}$.
If $y \in \scrL$, by Lem.~\ref{lem:jump-L-to-I},
$\phi(x) \sim \phi(y)$, and so $\phi(y) \in \scrD_{\phi(x)}$.
Conversely, if $I \in \scrD_{\phi(x)}$, $\phi(x) \sim I$,
so by Lem.~\ref{lem:jump-I-to-L}, there exists a $y \in \scrL$ for
which $I = \phi(y)$.  Thus $\scrD_{\phi(x)} = \im \phi$.

Since for any $y, z \in \scrL$, $\scrD_{\phi(y)} = \im \phi = \scrD_{\phi(z)}$,
$\scrD_{\phi(y)}$ is independent of the choice of $y \in \scrL$.
\end{proof}

\begin{definition} \label{def:D}
We define $\scrD$ to be the value of $\scrD_{\phi(x)}$ for any $x \in \scrL$.
\end{definition}

We introduce some machinery to prove all elements of $\scrP$
are reachable from the distinguished order ideal:

\begin{definition} \label{def:m}
Given $x \leq y$ in $\scrL$, we define
$m_{xy} : \scrL \rightarrow [x,y] : z \mapsto (z \wedge y) \vee x$.
\end{definition}

Intuitively, $m_{xy}$ compresses all of $\scrL$ into $[x,y]$.

\begin{lemma} \label{lem:m-modular}
Since $\scrL$ is distributive and thus is modular,
$m_{xy}(z) = (z \wedge y) \vee x = (z \vee x) \wedge y$.
\end{lemma}

\begin{lemma} \label{lem:m}
Given $x \leq y$ in $\scrL$, $p$ a prime filter of $\scrL$
with $x \not\in p$ and $y \in p$,
and $z \in \scrL$, then
$z \in p$ iff $m_{xy}(z) \in p$.
\end{lemma}
\begin{proof} \disconnect
Regarding $\Rightarrow$:
Suppose $z \in p$.  Since $y \in p$, $z \wedge y \in p$, and so
$m_{xy}(z) = (z \wedge y) \vee x \geq z \wedge y \in p$,
showing $m_{xy}(z) \in p$.

Regarding $\Leftarrow$:
Suppose $m_{xy}(z) = (z \wedge y) \vee x \in p$.
Since $p$ is a prime filter, either $x \in p$ or $z \wedge y \in p$.
But $x \not\in p$, so $z \wedge y \in p$.
So $z \geq z \wedge y \in p$, showing $z \in p$.
\end{proof}

\begin{lemma} \label{lem:pf-finite}
Given $x, y \in \scrL$, the set of prime filters which
contain $x$ and do not contain $y$ is finite.
\end{lemma}
\begin{proof} \disconnect
Define $P$ to be the set of prime filters which
contain $x$ and do not contain $y$.
Define $u = x \wedge y$ and $v = x \vee y$, so every filter in $P$
excludes $u$ and contains $v$.

Consider any two distinct elements $p$ and $q$ of $P$.
Necessarily there is a $z \in \scrL$ which is a member of one but not
the other; say $z \in p$ and $z \not\in q$.
By Lem.~\ref{lem:m}, $m_{uv}(z) \in p$ and $m_{uv}(z) \not\in q$.
Since $m_{uv}(z) \in [u,v]$,
$p \cap [u,v]$ and $q \cap [u,v]$ are distinct.
Thus the map $: P \rightarrow 2^{[u,v]} : p \mapsto p \cap [u,v]$
is one-to-one.

Since $\scrL$ is locally finite, $[u,v]$ is finite and $2^{[u,v]}$ is
finite, showing that $P$ is finite.
\end{proof}

\begin{lemma} \label{lem:reachable}
Given a locally-finite lattice $\scrL$,
every element $p$ in $\scrP$ is reachable from the \simclass
$\scrD = \im \phi$.
\end{lemma}
\begin{proof} \disconnect
Choose any $x \in \scrL$.  Then $\scrD$ is the \simclass
containing $\phi(x)$, and
in light of Lem.~\ref{lem:reach-equiv}, it
suffices to show that every $p \in \scrP$ is reachable from
$\phi(x) = \{ q \in \scrP \mvert x \in q \}$.
By Lem.~\ref{lem:reach}, we need to show that
$R = \PI{p} \setminus (\phi(x) \cap \PI{p})$ and
$S = \phi(x) \cap \PF{p}$ are finite.

Since $p$ is a prime filter of $\scrL$, $p$ is neither $\scrL$
nor $\myemptyset$, so there exist $y, z \in \scrL$ with
$y \not\in p$ and $z \in p$.

The elements of $R$ are
prime filters $r$ for which $r \leq p$, that is,
$p \subset r$, and $r \not\in \phi(x)$, that is, $x \not\in r$.
Since $z \in p$, $z \in r$, and thus every $r$ has
$z \in r$ and $x \not\in r$.
Lem.~\ref{lem:pf-finite} shows that the set of prime filters $r$
with $z \in r$ and $x \not\in r$ is finite, and thus $R$ is finite.

The elements of $S$ are
prime filters $s$ for which $s \geq p$, that is,
$s \subset p$, and $s \in \phi(x)$, that is, $x \in s$.
Since $y \not\in p$, $y \not\in s$, and thus every $s$ has
$y \not\in s$ and $x \in s$.
Lem.~\ref{lem:pf-finite} shows that the set of prime filters $s$
with $x \in s$ and $y \not\in s$ is finite, and thus $S$ is finite.

Thus $p$ is reachable from $\phi(x)$.
\end{proof}

\begin{theorem} \label{th:forward}
$\scrP$ is a poset and $\scrD$ is a \simclass of $\scrI$.
Every element of $\scrP$ is reachable from $\scrD$.
\end{theorem}
\begin{proof} \disconnect
  $\scrP$ is a poset by construction.
  Lems.~\ref{lem:jump-L-to-I} and~\ref{lem:jump-I-to-L} show that
  $\im \phi = \scrD$ is a \simclass.
  Lem.~\ref{lem:reachable} shows that every element of $\scrP$ is
  reachable from $\scrD$.
\end{proof}

\begin{theorem} \label{th:forward-reverse}
  $\phi$ is a lattice isomorphism between $\scrL$ and $\scrD$.
  Thus $\scrL$ is represented by the pair $(\scrP, \scrD)$ by
  being isomorphic to $\scrD$.
  The representation followed by the reconstruction
  takes $\scrL$ into $\scrD$, which are isomorphic to each other.
\end{theorem}
\begin{proof}
  The representation takes $\scrL$ into $(\scrP, \scrD)$,
  which the reconstruction takes into $\scrD$.
  Lem.~\ref{lem:lf-repr} shows that $\phi$ is a lattice isomorphism
  from $\scrL$ to $\scrD$ and thus they are isomorphic.
\end{proof}

\subsection{Considering \texorpdfstring{$\scrI$}{I} as a graph} \label{sec:graph}

We can look at $\scrI$ as a graph whose vertices are the elements of
$\scrI$ with two elements being connected by an edge iff one covers
the other.
Since $\scrI$ is the lattice of order ideals of a poset, one element
covers another iff their symmetric difference contains exactly one
element.  Thus two elements are connected by an edge iff their
symmetric difference contains exactly one element.
The graph can be partitioned into connected components, which are the
\simclasses of $\scrI$.

\begin{theorem} \label{th:component}
If $I, J \in \scrI$, $I \sim J$ iff $I$ and $J$ are
connected in the graph of $\scrI$.
\end{theorem}
\begin{proof} \disconnect
If $I \Symdiff J$ is finite, then, as in the proof of
Lem.~\ref{lem:jump-I-to-L},  we can construct a chain of elements of $\scrI$,
$I = K_0, K_1, K_2, \ldots, K_n = I \cap J,
K_{n+1}, K_{n+2}, \ldots, K_m = J$,
where for any $0 \leq i < m$, $|K_i \Symdiff K_{i+1}| = 1$.
Since $\scrI$ is the lattice of ideals of $\scrP$,
for any $i$, either $K_i \lessdot K_{i+1}$ or $K_i \gtrdot K_{i+1}$,
and thus $K_i$ and $K_{i+1}$ are connected by an edge in the graph.
Thus $I = K_0$ and $J = K_m$ are connected in the graph.

Conversely, if $I$ and $J$ are connected in the graph, there is a sequence of
elements of $\scrI$,
$I = L_0, L_1, L_2, \ldots, L_p = J$, where for any $0 \leq i < p$,
$L_i$ and $L_{i+1}$ are connected by an edge.
Thus, either $L_i \lessdot L_{i+1}$ or $L_i \gtrdot L_{i+1}$,
and since $\scrI$ is the lattice of ideals of $\scrP$,
$|L_i \Symdiff L_{i+1}| = 1$.
Thus $I \Symdiff J = L_0 \Symdiff L_p \subset
\bigcup_i L_i \Symdiff L_{i+1}$ is finite.
\end{proof}

\begin{corollary} \label{cor:component}
$\scrD$ is exactly one connected component of the graph of $\scrI$ and so
$\scrL$ is isomorphic to one connected component of the graph of $\scrI$.
\end{corollary}

\subsection{Example representations} \label{sec:example-repr}

In the example $\scrL = \mbbZ \times \mbbZ$ of Fig.~\ref{fig:ZxZ},
its dual is the pair of the poset $\scrP$ of Fig.~\ref{fig:ZxZprimes}
and the central \simclass $\scrD$ in Fig.~\ref{fig:ZxZidealsA}.
All elements of $\scrL$ are reachable from $\scrD$.
$\scrL$ is thus isomorphic to $\scrD$.
Applying the reconstruction to the pair $(\scrP, \scrD)$,
we derive a lattice (which is actually $\scrD$)
which is isomorphic to $\scrL$, as expected.

Now consider that there are eight other \simclasses in the
lattice of ideals in Fig.~\ref{fig:ZxZidealsA}.
None of these can be paired with $\scrP$ because for all of them,
some elements of $\scrP$ are not reachable.
Indeed, if we consider such a \simclass as the starting lattice
$\scrL_1$ of the representation, the resulting poset of prime
filters $\scrP_1$ is (naturally isomorphic to) a proper subposet of
$\scrP$.

However, other examples behave differently.  If we choose
$\scrL = \Bfin$, the lattice of finite subsets of $\mbbN$, then the
corresponding $\scrP$ is an antichain with $\aleph_0$ elements,
with each prime filter
being all the finite sets containing some particular integer,
which are all principal primes.
The corresponding $\scrI$ is isomorphic to $\mbbB_\mbbN$, the lattice
of \emph{all} subsets of $\mbbN$.  Then $\scrL$ is isomorphic to the
\emph{bottom} \simclass of $\scrI$, containing its
$\zerohat = \myemptyset$.

The top \simclass of this $\scrI$ contains its $\onehat =
\scrP$ and is isomorphic to
$\Bcofin$, the lattice of all cofinite subsets of $\mbbN$, which is
the lattice dual of $\Bfin$.
All other \simclasses are isomorphic to
$\Bfin \times \Bcofin$.  And since all \simclasses are countable but
$|\mbbB_\mbbN| = \mathfrak{c}$, there are $\mathfrak{c}$
\simclasses of $\scrI$.

If we chose $\scrL_1$ to be any \simclass other than $\Bfin$,
because all elements of $\scrP$ are reachable from any \simclass,
the corresponding $\scrP_1$
is still an antichain with $\aleph_0$ elements (naturally
corresponding to the elements of $\scrP$).
Differently from the bottom \simclass, its elements
are all secondary primes of $\scrL_1$.
The corresponding $\scrI_1$ is isomorphic to $\scrI$.

\subsection{The reconstruction} \label{sec:reverse}

In this section,
we will show two things:  First is that the reconstruction is
valid, in that it converts a pair into a locally-finite lattice.  The
second is that following the reconstruction by the representation
takes a pair into an isomorphic pair
(Def.~\ref{def:pair-isom}).

For this section, we fix the following:
Let $\scrP$ be the poset in the pair $(\scrP, \scrD)$.
Let $\scrI$ be its lattice of order ideals, ordered by inclusion.
Let $\scrD$ be the \simclass of $\scrI$ in the pair $(\scrP, \scrD)$.
Define $\scrL^\prime = \scrD$, the result of the reconstruction (the first
construction).
Let $\scrP^\prime$ be the poset of prime filters of
$\scrL^\prime$, ordered by inverse inclusion.
Let $\phi^\prime: \scrL^\prime \rightarrow \scrI^\prime:
x \mapsto \{p^\prime \in \scrP^\prime \mvert x \in p^\prime\}$ be the
representation
isomorphism in the representation construction.
Let $\scrD^\prime = \im \phi^\prime \subset \scrI^\prime$.
Thus the result of the representation (the second construction) is
$(\scrP^\prime, \scrD^\prime)$.

The scheme of the reconstruction is shown in the left
two parts of Fig.~\ref{fig:PIP}.
The scheme of the representation is shown in the right
three parts of the figure.
The overall scheme for applying, to a
pair of a poset and its \simclass, first the reconstruction
and then the representation, is shown by the figure overall.

\begin{figure}[ht]
\centering
\includegraphics[page=\ipeFigPILPI,scale=0.65]{locally-finite-figs.pdf}%
\caption{The overall scheme of the reconstruction
  (of $(\scrP, \scrD)$
  into $\scrL^\prime$) followed by the representation
  (of $\scrL^\prime$ into $(\scrP^\prime, \scrD^\prime)$):
  $\scrP$ is the given poset,
  $\scrI$ is the lattice of order ideals of $\scrP$,
  $\scrD = \scrL^\prime$ is the distinguished \simclass,
  $\scrP^\prime$ is the poset of prime filters of $\scrL^\prime$
  and the isomorphic image of $\scrP$.
  $\phi^\prime$ is the representation isomorphism from $\scrL^\prime$
  to $\scrD^\prime$:
  $\phi^\prime(x) = \{p^\prime \in \scrP^\prime \mvert x \in p^\prime\}$.
  $\psi$ is the pair isomorphism from $(\scrP, \scrD)$ to
  $(\scrP^\prime, \scrD^\prime)$:
  $\psi(p) = \PF{\PI{p}} \cap \scrD$.}
\label{fig:PIP}
\end{figure}

\begin{theorem} \label{th:converse-repr}
$\scrD$ is a locally-finite distributive lattice.
\end{theorem}
\begin{proof} \disconnect
It is well-known that the lattice of order ideals of a poset is a
distributive lattice.
Lem.~\ref{lem:DP-sublattice} shows that $\scrD$ is a sublattice of
$\scrI$ and so it is a distributive lattice.
What remains to be shown is that $\scrD$ is locally finite.
Let $J \leq K$ in $\scrD$, that is, $J \subset K$.
Since $J \sim K$, $K \setminus J = J \Symdiff K$ is finite.
We show that $[J, K]$ has a finite number of elements:
For any $L$ with $J \leq L \leq K$ in $\scrI$,
$J \subset L \subset K$, so
$L \setminus J \subset K \setminus J$.
Since $L = J \cup (L \setminus J)$,
there are at most $2^{|K \setminus J|}$ distinct possible values for $L$,
and that number is finite.  Thus $[J, K]$ has a finite number of elements
and so $\scrD$ is locally finite.
\end{proof}

We will now show that the representation theorem
Th.~\ref{th:forward} and the reconstruction representation theorem
Th.~\ref{th:converse-repr} are inverses up to isomorphism.
We define $\psi: \scrP \rightarrow \scrP^\prime$, which we will show
is the pair isomorphism (Def.~\ref{def:pair-isom}) from
from $(\scrP, \scrD)$ to $(\scrP^\prime, \scrD^\prime)$.
Thus,
we need to show that $\psi$ is a poset isomorphism from $\scrP$ to
$\scrP^\prime$ and that
$\psi_{\#\#}$ carries the \simclass $\scrD$ to the \simclass $\scrD^\prime$.

\begin{definition} \label{def:psi}
Given $p \in \scrP$,
necessarily $\PI{p} \in \scrI$ and so $\PF{\PI{p}} \subset \scrI$.
We define
$\psi: \scrP \rightarrow 2^\scrD: p \mapsto \PF{\PI{p}} \cap \scrD$.
\end{definition}

\begin{lemma} \label{lem:psi-pf}
Given $p \in \scrP$,
$\psi(p)$ is a prime filter of $\scrD$ and thus an element of
$\scrP^\prime$.
\end{lemma}
\begin{proof} \disconnect
Define $P = \PI{p}$, so $\psi(p) = \PF{P} \cap \scrD$.

Regarding $\psi(p) \neq \myemptyset$:
Since $p$ is reachable from $\scrD$,
use Def.~\ref{def:reach} to choose an ideal $I$ such that $I \in \scrD$
and $p \in I$.
Then $P \subset I$, $P \leq I$, $I \in \PF{P}$.
The existence of $I$ shows that
$\psi(p) = \PF{P} \cap D \neq \myemptyset$.

Regarding $\psi(p) \neq \scrD$:
Using Def.~\ref{def:reach}, choose $I$ such that $I \in \scrD$
but $p \not\in I$.  Since $p \in P$, $P \not\subset I$,
$P \not\leq I$, $I \not\in \PF{P}$.
The existence of $I$ shows that
$\psi(p) = \PF{P} \cap D \neq \scrD$.

Regarding $\psi(p)$ is closed upward:
Suppose $I, J \in \scrD$ with $I \in \psi(p)$ and $I \leq J$.
Then $I \subset J$.
Because $I \in \psi(p)$, $I \in \PF{P}$,
$I \geq P$, and $P \subset I$.
Since $I \subset J$, $P \subset J$, so
$J \in \PF{P}$.  Together with $J \in \scrD$, that shows
$J \in \psi(p)$.

Regarding $\psi(p)$ is closed under meets:
Suppose $I, J \in \scrD$ with $I,J \in \psi(p)$.
As previously, $P \subset I$ and $P \subset J$, implying
$P \subset I \cap J = I \wedge J$.
Since $\scrD$ is a sublattice of $\scrI$, $I \wedge J \in \scrD$.
so finally $I \wedge J \in \psi(p)$.

Regarding $\psi(p)$ is prime:
Suppose $I, J \in \scrD$ with $I \vee J \in \psi(p)$.
Then $P \subset I \vee J = I \cup J$.
This implies $p \in I \cup J$, so either $p \in I$ or
$p \in J$.
If $p \in I$, $P \subset I$, which implies $I \in \psi(p)$.
Similarly, if $p \in J$, $J \in \psi(p)$.
Thus, either $I \in \psi(p)$ or $J \in \psi(p)$, which shows that
$\psi(p)$ is prime.

Together, these show that $\psi(p)$ is a prime filter of
$\scrD$ and so is an element of $\scrP^\prime$.
\end{proof}

\begin{lemma} \label{lem:psi-princ}
Given $p \in \scrP$ and $P = \PI{p}$,
if $P \in \scrD$, then $\psi(p)$ is a principal filter of $\scrD$
and if $P \not\in \scrD$, then $\psi(p)$ is a secondary filter of $\scrD$.
\end{lemma}
\begin{proof} \disconnect
If $P \in \scrD$, then $\psi(p) = \{I \in \scrD \mvert P \leq I \}$
which is the principal filter of $\scrD$ generated by $P$.

Assume $P \not\in \scrD$.  We will show that $\psi(p)$ is secondary
by contradiction, so assume
$\psi(p)$ is a principal filter of
$\scrD$ generated by some $Q \in \scrD$, that is,
$\psi(p) = \{I \in \scrD \mvert Q \leq I \}$.
Since $Q \in \psi(p)$, $Q \in \PF{P}$, and $P \leq Q$.
But since $P \neq Q$, $P < Q$,
so there exists $q \in \scrP$ with
$q \in Q$ and $q \not\in P$.
Since $(\scrP, \scrD)$ is a pair, $q$ is reachable from $\scrD$.
Apply Lem.~\ref{lem:reach-messy} to show there exists
$R \in \scrD$ with $P \leq R < Q$.
Then $R \in \psi(p)$ but $R \not\in \PF{Q}$, showing
$\psi(p) \neq \PF{Q} \cap \scrD$.
Thus there is no such $Q$ and $\psi(p)$ is secondary.
\end{proof}

\begin{lemma} \label{lem:psi-isom-onto}
For any $p^\prime \in \scrP^\prime$, that there exists
some $p \in \scrP$ such that $p^\prime = \psi(p)$.
\end{lemma}
\begin{proof} \disconnect
Since $p^\prime \in \scrP^\prime$,
it is a prime filter of $\scrL^\prime$, and so $p^\prime$ is non-empty and each
element of $p^\prime$ is an ideal of $\scrP$.  Construct the intersection
of those ideals, $I = \bigcap p^\prime$, which is an element of $\scrI$.
We will show that $p^\prime = \PF{I} \cap \scrD$ and that
$I$ has a maximal element $p$, which implies $I = \PI{p}$.
Together these show that
$p^\prime = \PF{I} \cap \scrD = \PF{\PI{p}} \cap \scrD = \psi(p)$.

Regarding $p^\prime = \PF{I} \cap \scrD$:

Consider any $J \in \PF{I} \cap \scrD$, equivalently
$J \in\scrD$ and $I \leq J$.
Select some element $D$ of $p^\prime$, which must exist since
$p^\prime$ is a lattice filter of $\scrD$, so $D \in \scrD$ and
$I \subset D$.
Construct $K = D \wedge J$.
Since $\scrD$ is a sublattice of $\scrI$ by Lem.~\ref{lem:DP-sublattice},
$K \in \scrD$.

By construction, $I \subset K \subset D$.
Since $\scrD$ is a \simclass of $\scrD$, $K \Symdiff D$ is finite, so
we can label the elements of the difference:
$D = K \sqcup \{m_1, \ldots, m_n\}$ for some $n \geq 0$.

For any $1 \leq i \leq n$, $m_i \not\in K \supset I = \bigcap p^\prime$,
so there
exists a $P_i \in p^\prime$ for which $m_i \not\in P_i$.
If $n > 0$, construct $Q = \bigwedge_i (D \wedge P_i) \in \scrD$ and
if $n = 0$, construct $Q = D \in \scrD$.%
\footnote{Each $P_\bullet \in p^\prime$ and $p^\prime$ is a prime
filter of $\scrL$. But we cannot use the convention that a meet of
zero $P_\bullet = \onehat_{p^\prime}$ because $p^\prime$ may not have
a maximum.
We can correct the formalism by defining
$Q = \bigwedge_i (D \wedge P_i)$, with the $\bigwedge$ being taken in
the lattice $p^* = \{x \in p^\prime \mvert x \leq D\}$, since
$\onehat_{p^*}$ exists and $= D$, so if $n=0$,
$Q = \bigwedge \myemptyset = \onehat_{p^*} = D$.}
Since $p^\prime$ is a lattice filter, $Q \in p^\prime$.
By construction, $Q \subset K$, which, since $p^\prime$ is a filter,
implies $K \in p^\prime$,
which in turn shows $J \in p^\prime$.

Conversely, consider any $J \in p^\prime$.
Then, since $p^\prime$ is a filter in $\scrD$, $J \in \scrD$.
By construction, $I = \bigcap p^\prime \subset J$,
so $I \leq J$ and $J \in \PF{I}$.
Thus $J \in \PF{I} \cap \scrD$.

This shows $p^\prime = \PF{I} \cap \scrD$.

Regarding $I$ has a maximum element $p$ and thus $I = \PI{p}$:

We show every element $q$ of $I$ is $\leq$ some maximal element of $I$:
Choose some $D \in p^\prime$, so necessarily $q \in D$.
$q$ is reachable from $\scrD$, implying $q$ is reachable from $D$ and,
by Lem.~\ref{lem:reach}, $\PF{q} \cap D$ is finite.
Since $I \subset D$, $\PF{q} \cap I$ is finite.
There must be a maximal element $q^+$ of
$\PF{q} \cap I$ and so $q^+ \geq q$.
But $q^+$ must be maximal in $I$, since if $q^{++} \in I$ and $q^{++} > q^+$,
$q^{++} \in \PF{q} \cap I$ and then $q^+$ is not maximal in $\PF{q} \cap I$.

We show there is only one maximal element of $I$:

Suppose $r_1$ and $r_2$ are maximal in $I$ and $r_1 \neq r_2$.
Construct the ideals $I_1 = I \setminus \{r_2\}$ and
$I_2 = I \setminus \{r_1\}$.
Clearly $I = I_1 \cup I_2$; $I_1, I_2 \subset I$; and
$I_1, I_2 < I$ in $\scrI$.

Both $r_1$ and $r_2$ are reachable from $\scrD$ and clearly neither is
$\leq$ the other.  Thus by Lem.~\ref{lem:reach-include-exclude} there
exists an $L_1 \in \scrD$ with $r_1 \in L_1$ and $r_2 \not\in L_1$.
Construct $Q_1 = L_1 \vee I_1$.  Trivially, $r_1 \in Q_1$ and
$r_2 \not\in Q_1$.
Choosing an arbitrary $D \in p^\prime$, $I \subset D$ and
$L_1 \leq L_1 \vee I_1 = Q_1 \leq L_1 \vee I \leq L_1 \vee D$.
Since both $L_1$ and $D$ are in the convex sublattice $\scrD$,
$Q_1 \in \scrD$.

Similarly, by Lem.~\ref{lem:reach-include-exclude} there
exists an $L_2 \in \scrD$ with $r_2 \in L_2$ and $r_1 \not\in L_2$.
Construct $Q_2 = L_2 \vee I_2$.
Then $r_1 \not\in Q_2$, $r_2 \in Q_2$, and $Q_2 \in \scrD$.

Since $r_2 \not\in Q_1$ and $r_2 \in I = \bigcap p^\prime$,
$Q_1\not\in p^\prime$.
Similarly, since $r_1 \not\in Q_2$ and $r_1 \in I$,
$Q_2\not\in p^\prime$.
Since $Q_1, Q_2 \in \scrD$, $Q_1 \vee Q_2 \in \scrD$.
Then $Q_1 \vee Q_2 \geq I_1 \vee I_2 = I$, so
$Q_1 \vee Q_2 \in \PF{I} \cap \scrD = p^\prime$.
This contradicts that $p^\prime$ is a prime filter.
Hence there are no two such $r_1$ and $r_2$.

Since $I$ is an ideal, every element of $I$ is $\leq$ some maximal
element, and $I$ has only one maximal element,
there exists some $p \in I$ that is maximum, with $I = \PI{p}$.
This shows $p^\prime = \PF{\PI{p}} \cap \scrD$ for that $p$.
\end{proof}

\begin{lemma} \label{lem:psi-isom}
$\psi$ is a poset isomorphism from $\scrP$ to $\scrP^\prime$.
\end{lemma}
\begin{proof} \disconnect
Let $p, q \in \scrP$, $P = \PI{p}$, and $Q = \PI{q}$.

Regarding that $p \leq q$ in $\scrP$ implies $\psi(p) \leq \psi(q)$:
If $p \leq q$ then $P \subset Q$ and $P \leq Q$.
That implies $\PF{Q} \subset \PF{P}$
and $\psi(q) = \PF{Q} \cap \scrD \subset \PF{P} \cap \scrD = \psi(p)$,
and since $\scrP^\prime$ is
ordered by inverse inclusion, $\psi(p) \leq \psi(q)$

Regarding that $\psi(p) \leq \psi(q)$ implies $p \leq q$:
Assume $p \not\leq q$.
Use Lem.~\ref{lem:reach-include-exclude} to construct an ideal $U \in \scrD$
with $p \not\in U$ and $q \in U$.
Since $U \in \scrD$ and $q \in U$, we have
$Q \subset U$, $Q \leq U$, $U \in \PF{Q}$, and $U \in \psi(q)$.
Similarly, since $U \in \scrD$ and $p \not\in U$, we have
$U \not\in \psi(p)$.
These show that $\psi(q) \not\subset \psi(p)$ which is equivalent to
$\psi(p) \not\leq \psi(q)$.
This is a contradiction and thus it must be that $p \leq q$.

Regarding $\psi$ is one-to-one:
The preceding argument shows that if $\psi(p) = \psi(q)$,
then $p = q$.

Regarding $\psi$ is onto:
This is shown by Lem.~\ref{lem:psi-isom-onto}
\end{proof}

\begin{lemma} \label{lem:psi-phi}
As functions from $\scrL^\prime$ to $\scrI^\prime$,
$\psi_{\#} = \phi^\prime$.
\end{lemma}
\begin{proof} \disconnect
To prove this, we need to show that for any
for $x \in \scrL^\prime$, $\phi^\prime(x) = \psi_\#(x)$.
That is equivalent to showing
$\{p^\prime \in \scrP^\prime \mvert x \in p^\prime\} = \{\psi(y) \mvert y \in x\}
= \{\PF{\PI{y}} \cap \scrD \mvert y \in x\}$.
We show this in two parts:

Regarding that for any $p^\prime \in \scrP^\prime$ with $x \in p^\prime$,
there exists a $y \in x$ for which
$\psi(y) = p^\prime$:
The required $y$ is $\psi^{-1}(p^\prime)$, which exists and is
$\in \scrP$ because $\psi$ is a poset isomorphism between
$\scrP$ and $\scrP^\prime$ by Lem.~\ref{lem:psi-isom};
what remains to be shown is that $y \in x$.
By construction, $p^\prime = \psi(y) = \PF{\PI{y}} \cap \scrD$.
Since $x \in p^\prime$, $x \in \PF{\PI{y}}$, $x \geq \PI{y}$,
$\PI{y} \subset x$.
Since $x$ is an ideal of $\scrP$, $y \in x$.

Regarding that for any $y \in x$, there exists a
$p^\prime \in \scrP^\prime$ with $x \in p^\prime$ for which
$\PF{\PI{y}} \cap \scrD = p^\prime$:
The required $p^\prime$ is simply $\PF{\PI{y}} \cap \scrD$;
what remains to be shown is that $x \in p^\prime$.
Since $y \in x$ and $x$ is an ideal of $\scrP$, $\PI{y} \subset x$.
$x \geq \PI{y}$, and so $x \in \PF{\PI{y}}$.
Trivially, $x \in \scrD$, so $x \in \PF{\PI{y}} \cap \scrD = p^\prime$.
\end{proof}

\begin{lemma} \label{lem:psi-D}
  $\psi_{\#\#}(\scrD) = \scrD^\prime$.
\end{lemma}
\begin{proof} \disconnect
Because $\phi^\prime$ is the representation isomorphism from
$\scrL^\prime$ to $\scrD^\prime$, $\scrD = \scrL^\prime$,
and $\psi_\# = \phi^\prime$,
$\psi_\#$ is a bijection from $\scrD$ to $\scrD^\prime$.
The lemma follows from that.
\end{proof}

\begin{theorem} \label{th:reverse-forward}
  $\psi$ is a pair-isomorphism from $(\scrP, \scrD)$ to
  $(\scrP^\prime, \scrD^\prime)$.
\end{theorem}
\begin{proof}
Given the definition of pair-isomorphism, Def.~\ref{def:pair-isom},
Lem.~\ref{lem:psi-isom} and Lem.~\ref{lem:psi-D} show that
$\psi$ is a pair-isomorphism from $(\scrP, \scrD)$ to
$(\scrP^\prime, \scrD^\prime)$.
\end{proof}

It is interesting to examine how the reconstruction
construction operates on ``defective'' pairs $(\scrP, \scrD)$, where
$\scrD$ is a \simclass of $\scrI$ but not all elements of $\scrP$ are
reachable from $\scrD$.

\begin{theorem} \label{th:defective}
Given a poset $\scrP$, constructing $\scrI$, $\scrL^\prime$,
$\scrP^\prime$, and $\psi$ as above, and given a \simclass $\scrD$ of
$\scrI$, the elements of $\scrP$ divide into three classes:
\begin{enumerate}
\item $p$ which are reachable from $\scrD$:
  $\psi(p)$ is an element of $\scrP^\prime$.
  We call this class $\mathcal{R}$, and it is a convex subposet of $\scrP$.
\item $p$ which are elements of every ideal in $\scrD$:
  $\psi(p) = \scrD$ and thus is not a prime filter of
  $\scrL^\prime = \scrD$ and is not an element of $\scrP^\prime$.
  We call this class $\mathcal{U}$.
\item $p$ which are elements of no ideal in $\scrD$:
  $\psi(p) = \myemptyset$ and thus is not a prime filter of
  $\scrL^\prime = \scrD$ and is not an element of $\scrP^\prime$.
\end{enumerate}
$\psi$ is a poset isomorphism from $\mathcal{R}$ to $\scrP^\prime$
and thus $\scrP^\prime$ is isomorphic to $\mathcal{R}$.
\end{theorem}
\begin{proof} \disconnect
Lem.~\ref{lem:reach-convex} shows $\mathcal{R}$ is a convex
subposet of $\scrP$.

To show that $\psi$ is a poset isomorphism from $\mathcal{R}$ to
$\scrP^\prime$:

From $\mathcal{R}$, construct $\scrI_{\mathcal{R}}$, $\scrL^\prime_{\mathcal{R}}$,
$\scrP^\prime_{\mathcal{R}}$.
Construct $\scrD_{\mathcal{R}} = \{D \setminus \mathcal{U} \mvert D \in \scrD\}$
and then construct $\psi_{\mathcal{R}}$ as above.
By Lem.~\ref{lem:psi-isom}, $\psi_{\mathcal{R}}$ is a poset isomorphism from
$\mathcal{R}$ to $\scrP^\prime_{\mathcal{R}}$.

Of course, $\mathcal{R}$ is embedded in $\scrP$.  By extension,
$P \mapsto P \cup \mathcal{U}$, that is, adding to each ideal the set
$\mathcal{U}$, is a natural poset monomorphism of
$\scrI_{\mathcal{R}}$ into $\scrI$.
Similarly, $P \mapsto P \cup \mathcal{U}$ is a natural poset
monomorphism of $\scrD_{\mathcal{R}}$ into $\scrD$.
And then, $D^\prime \mapsto \{ P \cup \mathcal{U} \mvert P \in
D^\prime\}$, that is, adding $\mathcal{U}$ to each element of such a
$D^\prime$, is a natural poset monomorphism of
$\scrP^\prime_{\mathcal{R}}$ into $\scrP^\prime$.

The composition of $\psi_{\mathcal{R}}$ with the monomorphism of
$\scrP^\prime_{\mathcal{R}}$ into $\scrP^\prime$ is $\psi$.
\end{proof}

\subsection{An alternative duality} \label{sec:alternative}

One category of the duality in Th.~\ref{th:duality} is (isomorphism
classes of) pairs of posets
with a \simclass of order ideals thereof.  An alternative duality exists
where that category is replaced by the simpler category of
pairs of posets with an order ideal thereof
--- a representative element of the \simclass --
but it requires extending
pair-isomorphism to include that if two ideals of $\scrP$, $D$ and
$D^\prime$, have $D \sim D^\prime$, then $(\scrP, D)$ and
$(\scrP, D^\prime)$ are isomorphic:

\begin{definition} \label{def:alt-pair}
An \emph{alternate-pair} is a pair of
\begin{enumerate}
\item a poset and
\item a distinguished order ideal of the poset
\end{enumerate}
for which all elements of the poset are reachable from the ideal
  (Def.~\ref{def:reach}).
\end{definition}

\begin{definition} \label{def:alt-pair-isom}
We define that
two alternative-pairs $(\scrP, D)$ and $(\scrP^\prime, D^\prime)$ of
posets with respective order ideals are \emph{isomorphic}
iff there exists a poset isomorphism $\psi: \scrP \rightarrow \scrP^\prime$
such that $\psi_\#(D) \sim D^\prime$,
where $\psi_\#$ is the image map of $\psi$.
\end{definition}

In the alternate representation category,
every pair $(\scrP, \scrD)$ is replaced by all of
the alternative-pairs $(\scrP, D)$ for every $D \in \scrD$, all of
which are isomorphic to each other.
Then we define the duality:

\begin{definition} \label{def:alt-forward}
The representation of the alternative duality is:
From the locally-finite lattice $\scrL$ we construct its poset of
prime filters $\scrP$ (ordered by inverse inclusion).
From $\scrP$ we construct its lattice of order
ideals $\scrI$ (ordered by inclusion).
We choose an arbitrary element $x \in \scrL$, and construct
$D = \{p \in \scrP \mvert x \in p\} = \phi(x)$.
The resulting alternate-pair is $(\scrP, D)$.
(All choices of $x$ result in isomorphic alternate-pairs.)
\end{definition}

\begin{definition} \label{def:alt-reverse}
The reconstruction of the alternative duality is:
From the pair of a poset $\scrP$ and its order ideal $D$,
we construct from $\scrP$ its lattice of order
ideals $\scrI$ (ordered by inclusion).
We then extract $\scrD$, which is the \simclass of $\scrI$ containing $D$,
which is necessarily a locally-finite lattice.
\end{definition}

\begin{theorem} \label{th:alt-duality}
The alternative representation and reconstruction form a duality.
\end{theorem}
\begin{proof}[Proof sketch] \disconnect
The alternative representation is derivable from the primary
representation by, on the representation side, replacing each
isomorphism class of pairs of posets and \simclasses with an
isomorphism class of alternate pairs.  In this replacement, each pair
$(\scrP, \scrD)$ in an isomorphism class is replaced by all of the
alternate pairs $(\scrP, D)$ for all $D \in \scrD$.

Applying this translation to the proof of Th.~\ref{th:duality} gives
the theorem.
\end{proof}

\section{Criterion for secondary primes} \label{sec:criterion}

In this section, $\scrL$ is a locally-finite distributive lattice.
If all of the prime filters of $\scrL$ are principal,
they correspond to the join-irreducible elements of $\scrL$, and so the
structure of $\scrP$ is relatively easy to determine.  But if a
prime is secondary, it has no single corresponding element of $\scrL$
and so identifying secondary primes is more complex.  This section
gives a criterion for recognizing secondary primes from the local
covering structure of $\scrL$.

\begin{definition} \cite{Birk1967a}*{\S I.7} \label{def:transpose}
Given two coverings $x^- \lessdot x^+$ and $y^- \lessdot y^+$ in $\scrL$,
we define that
$y^- \lessdot y^+$ is a \emph{downward transpose} of $x^- \lessdot x^+$,
in symbols $\transpose{x}{y}$,
iff $x^+ = x^- \vee y^+$ and $y^- = x^- \wedge y^+$.
\end{definition}

\begin{definition} \label{def:sep}
Given a covering $x^- \lessdot x^+$, Lem.~\ref{lem:diff-one}
shows $\phi(x^+) \setminus \phi(x^-)$ contains exactly one element.
We define the \emph{separator} of $x^-$ and $x^+$,
$x^+ \sslash x^-$, to be that unique element.
Thus $x^+ \sslash x^- \in \scrP$.
\end{definition}

Intuitively, the separator is the unique prime filter that
distinguishes the two elements of a covering.
Downward transposes preserve separators.

\begin{lemma} \label{lem:trans-pres}
If $\transpose{x}{y}$ then
$x^+ \sslash x^- = y^+ \sslash y^-$.
\end{lemma}
\begin{proof} \disconnect
We know that $\phi(x^+) \setminus \phi(x^-)$ and
$\phi(y^+) \setminus \phi(y^-)$ each contain exactly one element.
Define $p = x^+ \sslash x^-$, the element of $\phi(x^+) \setminus \phi(x^-)$,
so $p \not\in \phi(x^-)$, $x^- \not\in p$, and
$x^- \vee y^+ = x^+ \in p$.
Since $p$ is a prime filter, either $x^- \in p$ or $y^+ \in p$.
But $x^- \not\in p$, so $y^+ \in p$.
Since $x^- \not\in p$ and $y^- \leq x^-$,
$y^- \not \in p$.
Thus $p \in \phi(y^+) \setminus \phi(y^-)$, so $p = y^+ \sslash y^-$.
\end{proof}

\begin{lemma} \label{lem:trans-join}
Suppose $x^- \lessdot x^+$, $y^- \lessdot y^+$, and
$p = x^+ \sslash x^- = y^+ \sslash y^-$.
Defining $w^- = x^- \vee y^-$ and $w^+ = x^+ \vee y^+$,
$w^- \lessdot w^+$, $\transpose{w}{x}$, $\transpose{w}{y}$,
and $p = w^+ \sslash w^-$.
Dually, defining $z^- = x^- \wedge y^-$ and $z^+ = x^+ \wedge y^+$,
$z^- \lessdot z^+$, $\transpose{x}{z}$, $\transpose{y}{z}$,
and $p = z^+ \sslash z^-$.
\end{lemma}
\begin{proof} \disconnect
We know that $\phi(x^+) = \phi(x^-) \sqcup \{p\}$
and  $\phi(y^+) = \phi(y^-) \sqcup \{p\}$.
By Th.~\ref{th:phi-embed},  $\phi$ is a lattice embedding into
$\scrI$, so
$\phi(w^+) = \phi(x^+) \cup \phi(y^+)$ and
$\phi(w^-) = \phi(x^-) \cup \phi(y^-)$.
Thus $\phi(w^+) = \phi(w^-) \sqcup \{p\}$, which by
Lem.~\ref{lem:diff-n} means $w^- \lessdot w^+$.

Set computation shows that
$\phi(w^- \vee x^+) = \phi(w^+)$ and
$\phi(w^- \wedge x^+) = \phi(x^-)$, so
$w^- \vee x^+ = w^+$, $w^- \wedge x^+ = x^-$, and so
$\transpose{w}{x}$.
Similarly, we show $\transpose{w}{y}$.

Lem.~\ref{lem:trans-pres} shows that $p = w^+ \sslash w^-$.

Similarly, since $\phi$ is a lattice embedding into $\scrI$,
$\phi(z^+) = \phi(x^+) \cap \phi(y^+)$ and
$\phi(z^-) = \phi(x^-) \cap \phi(y^-)$.
Thus $\phi(z^+) = \phi(z^-) \sqcup \{p\}$, which by
Lem.~\ref{lem:diff-n} means $z^- \lessdot z^+$.

Again, set computation shows that
$\phi(x^- \vee z^+) = \phi(z^+)$ and
$\phi(x^- \wedge z^+) = \phi(x^-)$, so
$x^- \vee z^+ = z^+$, $x^- \wedge z^+ = x^-$, and so
$\transpose{x}{z}$.
Similarly, we show $\transpose{y}{z}$.

Lem.~\ref{lem:trans-pres} shows that $p = z^+ \sslash z^-$.
\end{proof}

Intuitively, the relationship
``$x^+ \sslash x^- = y^+ \sslash y^-$'' is the symmetric transitive
closure of the relationship ``$\transpose{x}{y}$''.
That is, if two covers are connected by any sequence of downward
transposes or reverse downward transposes,
they have the same separator.
Conversely, if two covers have the same separator, they are
both downward transposes of their join and
their meet is a downward transpose of both of them.
In the case where the upper elements of the covers are comparable, the
lesser
covering is the downward transpose of the greater; and similarly if the lower
elements are comparable.
This also implies that ``downward transpose'' is transitive.

A prime filter is always the separator of at least one covering.

\begin{lemma} \label{lem:exist-cover}
For any prime filter $p$, there is at least one covering
$x^- \lessdot x^+$ for which $p = x^+ \sslash x^-$.
\end{lemma}
\begin{proof} \disconnect
Since $p$ is a lattice filter, it is non-empty, exists a
$y \in \scrL$ for which
$y \in p$.  Since $p$ is a proper lattice filter, there exists a
$z \in \scrL$ for which $z \not\in p$.
Necessarily, $y \vee z \in p$ and $z \leq y \vee z$.
Since $\scrL$ is locally finite, there is a
finite saturated chain between $z$ and $y \vee z$,
$z = w_0 \lessdot w_1 \lessdot w_2 \lessdot \cdots \lessdot w_n = y \vee z$.
Since $w_0 \not\in p$ and $w_n \in p$, there exists an $0 \leq i < n$
for which $w_i \not\in p$ and $w_{i+1} \in p$.
Thus $p \not\in \phi(w_i)$ and $p \in \phi(w_{i+1})$, so
$p = w_{i+1} \sslash w_i$.
\end{proof}

For any covering, what downward transposes it has reflects the
properties of its separator.

\begin{lemma} \label{lem:down-step}
If $x^- \lessdot x^+$ and $p = x^+ \sslash x^-$, then
\begin{enumerate}
\item if $p$ is principal and $p = \PF{x^+}$, then if
  $y^- \lessdot y^+$ with $\transpose{x}{y}$, it must be that
  $y^+ = x^+$ and $y^- = x^-$.
\item otherwise, there exists a $y^- \lessdot y^+$ with
  $\transpose{x}{y}$ and $x^+ > y^+$.
\end{enumerate}
\end{lemma}
\begin{proof} \disconnect
Regarding (1):
For any $y^- \lessdot y^+$ with  $y^- \lessdot y^+$ and
$\transpose{x}{y}$, by
Lem.~\ref{lem:trans-pres}, $p = y^+ \sslash y^-$ and $y^+ \in p$.
But that means that $y^+ \geq x^+$, which together with
$\transpose{x}{y}$ implies $y^+ = x^+$, which implies $y^- = x^-$.

Regarding (2):
If $p$ is not principal, or $p \neq \PF{x^+}$, then
$x^+$ is not a minimal element of $p$:
If $x^+$ were minimal in $p$, then for any $z \in p$,
the fact that $p$ is a lattice filter means that
$x^+ \wedge z \in p$.  The minimality of $x^+$ implies that
$x^+ \wedge z \geq x^+$, which is $x^+ \leq z$.  Thus $x^+$ is
actually the minimum of $p$ and $p = \PF{x^+}$, which is a
contradiction.

Choose $y^+ \in p$ for which $y^+ < x^+$.
Then $y^+ \not\leq x^-$ because otherwise $y^+ \in p$ and $y^+ \leq x^-$
imply $x^- \in p$, which is not true.
So $x^- < x^- \vee y^+ \leq x^+$, forcing $x^- \vee y^+ = x^+$.
Define $y^- = y^+ \wedge x^-$.
Then by the diamond isomorphism theorem, $y^- \lessdot y^+$
and so $\transpose{x}{y}$.
\end{proof}

\begin{definition} \label{def:chain}
We define a \emph{chain} of downward transposes to be a sequence of coverings
$x^-_1 \lessdot x^+_1, x^-_2 \lessdot x^+_2, x^-_3 \lessdot x^+_3, \ldots,
x^-_n \lessdot x^+_n$ with
$x^+_1 > x^+_2 > x^+_3 > \cdots > x^+_n$
and $\transpose{x_1}{x_2}$, $\transpose{x_2}{x_3}$, $\ldots$,
$\transpose{x_{n-1}}{x_n}$.
\end{definition}

\begin{lemma} \label{lem:chain-alt}
In the context of Def.~\ref{def:chain}, the definition of a chain of downward
transposes, the condition
$x^+_1 > x^+_2 > x^+_3 > \cdots > x^+_n$
is equivalent to the condition
$x^-_1 > x^-_2 > x^-_3 > \cdots > x^-_n$.
\end{lemma}
\begin{proof} \disconnect
For every $1 \leq i < n$, by hypothesis
$x^+_i = x^-_i \vee x^+_{i+1}$ and $x^-_{i+1} = x^-_i \wedge x^+_{i+1}$.
Thus necessarily $x^-_i \geq x^-_{i+1}$ and $x^+_i \geq x^+_{i+1}$.
Given these equalities,
the properties $x^-_i = x^-_{i+1}$ and $x^+_i = x^+_{i+1}$ imply each
other.
Given these $\geq$ relations,
$x^-_i = x^-_{i+1}$ is the negation of $x^-_i > x^-_{i+1}$ and
$x^+_i = x^+_{i+1}$ is the negation of $x^+_i > x^+_{i+1}$.
Together, these show that
the properties $x^-_i > x^-_{i+1}$ and $x^+_i > x^+_{i+1}$ imply each
other.
\end{proof}

The preceding lemmas show that downward transposes reveal differences
between covers whose separators are principal and secondary primes.
Chains of downward transposes, and the conditions under which they can
be extended, form a test for which separators are principal and
secondary primes, and also can group the chains to correspond to (and thus
enumerate) the prime filters.

\begin{theorem} \label{th:chain}
If $x^- \lessdot x^+$ and $p = x^+ \sslash x^-$, then:
\begin{enumerate}
\item If $p$ is principal, then any chain of downward transposes
starting with $x^- \lessdot x^+$ can only be extended to some finite
maximum length, which depends on $x^- \lessdot x^+$ but not the chain
as a whole.
\item If $p$ is principal and $y^+$ is the minimum element of $p$,
$p = \PF{y^+}$, define $y^-$ as the unique element that $y^+$ covers.
Then any chain of downward transposes starting with $x^- \lessdot x^+$
that does not end in the covering $y^- \lessdot y^+$ can be extended
by appending $y^- \lessdot y^+$.
Any chain ending with $y^- \lessdot y^+$ can not be extended.
\item If $p$ is secondary, then any chain of downward transposes
starting with $x^- \lessdot x^+$ can be extended by one element,
and thus any such chain can be extended indefinitely.
\item Given additionally $y^- \lessdot y^+$ with $p = y^+ \sslash y^-$,
and two chains of downward transposes, one starting with
$x^- \lessdot x^+$ and one starting with $y^- \lessdot y^+$,
there is a single covering which, relative to each
chain, is either the final covering of the chain or can be appended to it.
Thus, both chains can be made to have the same final covering by
conditionally appending a downward transpose.
\end{enumerate}
\end{theorem}
\begin{proof} \disconnect
Regarding (1):
If $p$ is principal, there is a $y^+ \in \scrL$ such that $p = \PF{y^+}$.
Consider any chain of downward transposes
$x^- = x^-_1 \lessdot x^+ = x^+_1, x^-_2 \lessdot x^+_2, x^-_3 \lessdot x^+_3,
\ldots, x^-_n \lessdot x^+_n$.
Then for every $1 \leq i \leq n$, $x^+_i \leq x^+$, and repeated
application of Lem.~\ref{lem:trans-pres} shows $p = x^+_i \sslash x^-_i$.
Thus $p \in \phi(x^+_i)$, $x^+_i \in p$, and $x^+_i \geq y^+$.
Therefore the $x_\bullet$ contain no duplicate values and form a chain
in $[y^+,x^+]$,
and so by Lem.~\ref{lem:diff-n}, $n-1 \leq \rho(y^+,x^+)$
and $n \leq \rho(y^+,x^+)+1$, which is a finite value.

Regarding (2):
By Lem.~\ref{lem:PF-irred}, $y^+$ is join-irreducible, and by
Lem.~\ref{lem:cover-one}, $y^+$ covers exactly one element, which we
call $y^-$.

Regarding (2), a chain that does not end in $y^- \lessdot y^+$:
Continuing the argument regarding (1), we know that $x^+_n \geq y^+$.
If $x^+_n = y^+$, then $x_n^- \lessdot y^+$, and
since $y^+$ only covers $y^-$, $x_n^- = y^-$.
Then by Lem.~\ref{lem:chain-alt}, $x^+_n = y^+$ and so the last
transpose in the chain is $y^- \lessdot y^+$, which contradicts the
hypotheses.

Thus we know $x^+_n > y^+$.
But $p \in \phi(y^+)$ and $p \not\in \phi(x^-_n)$, showing
$x_n^- \not\geq y^+$.
Since $x^+_n$ covers $x^-_n$, it must be that $x^+_n = x^-_n \vee y^+$.
By the diamond isomorphism theorem, $x^-_n \wedge y^+$ must be covered
by $y^+$, which forces $y^- = x^-_n \wedge y^+$.
Together, those show that $\transpose{x_n}{y}$.

Regarding (2), a chain that ends in $y^- \lessdot y^+$:
Lem.~\ref{lem:down-step}(1) shows that the chain can not be extended.

Regarding (3):
Given any chain of downward transposes
$x^- = x^-_1 \lessdot x^+ = x^+_1, x^-_2 \lessdot x^+_2, x^-_3 \lessdot x^+_3,
\ldots, x^-_n \lessdot x^+_n$,
repeated application of Lem.~\ref{lem:trans-pres} shows
$p = x^+_n \sslash x^-_n$.
Then Lem.~\ref{lem:down-step} can be applied to $x^-_n \lessdot x^+_n$ and
$p = x^+_n \sslash x^-_n$ to obtain a
$y^- \lessdot y^+$ with $x^+_n > y^+$ and $\transpose{x_n}{y}$.
Thus $y^- \lessdot y^+$ can be appended to the chain to obtain a chain
with length $n+1$.
Clearly this can be repeated indefinitely for any chain meeting the
hypotheses.

Regarding (4):
Given any two such chains, one starting with $x^- \lessdot x^+$ and
ending with
$x_n^- \lessdot x_n^+$, and one starting with $y^- \lessdot y+$ and
ending with
$y_m^- \lessdot y_m^+$,
define $z^+ = x_n^+ \wedge y_m^+$ and $z^- = x_n^- \wedge y_m^-$
Lem.~\ref{lem:trans-join} shows that
$z^- \lessdot z^+$ that is a downward transpose of both
$x_n^- \lessdot x_n^+$ and $y_m^- \lessdot y_m^+$ and so for each
chain is either the final covering of the chain or can be appended to
extend the chain.
\end{proof}

\begin{definition} \label{def:inf-chain}
We define an \emph{indefinitely-extendable chain} to be a chain which
is an initial segment of a sequence of coverings
$x^-_1 \lessdot x^+_1, x^-_2 \lessdot x^+_2, x^-_3 \lessdot x^+_3, \ldots$
where every initial segment of the sequence is a chain.

We define a \emph{terminating chain} to be a
chain which cannot be extended beyond some finite length (which might
be the length of the chain itself).
\end{definition}

By Lem.~\ref{lem:exist-cover}, every prime filter is the separator of
at least one covering. Considering the chains that start from such a covering,
Th.~\ref{th:chain} allows us to use the chains to test whether the prime
filter is principal or secondary:
\begin{itemize}
\item If there exists such a terminating chain, the prime filter is principal;
\item If there exists such an indefinitely-extendable chain, the prime
filter is secondary.
\end{itemize}
Conversely, the nature of the filter controls the behavior of the chains:
\begin{itemize}
\item If the prime filter is principal, all such chains are terminating chains;
\item If the prime filter is secondary, all such chains are
  indefinitely-extendable.
\end{itemize}

The set of all chains can divided into groups that
are the maximal sets of chains that can all
pairwise merge as they are extended downward, which is equivalent to
grouping chains by the separators of their coverings.
If every member of a group, upon being repeatedly extended,
cannot be extended beyond a particular covering, the group corresponds
to the principal prime generated by the join-irreducible element that
is the top element of that limiting covering.
If every member of a group is indefinitely extendable,
the group corresponds to a secondary prime.

Often, discerning these properties of chains is simpler than analyzing
prime filters directly.
For example, in
Fig.~\ref{fig:ZxZchain} all chains of downward transposes in
$\mbbZ \times \mbbZ$ can be
extended indefinitely, so all prime filters are secondary.
There are two sets of secondary primes, one corresponding to chains
that proceed toward the lower left and one corresponding to chains
that proceed toward the lower right, with each set indexable by
$\mbbZ$.

But in Fig.~\ref{fig:SSchain} all chains of downward transposes
in 4-SkewStrip must end, showing all prime filters are principal.
The coverings at which chains are forced to terminate have
the join-irreducible elements of $\scrL$ as their upper elements.
Collectively, these join-irreducible
elements form a poset isomorphic to $\scrP$.
$\scrI$ has three connected components,
one is $\myemptyset$, one is $\scrP$, and one is the remaining ideals of
$\scrP$, all of which are infinite and have complements that are infinite.
The third connected component of $\scrI$ is isomorphic to 4-SkewStrip.

\begin{figure}[ht]
\centering
\begin{subfigure}{0.45\textwidth}
  \centering
  \includegraphics[page=\ipeFigZxZchain]{locally-finite-figs.pdf}%
  \caption{The lattice $\mbbZ \times \mbbZ$ with two chains of downward
    transposes}
  \label{fig:ZxZchain}
\end{subfigure}
\begin{subfigure}{0.45\textwidth}
  \centering
  \includegraphics[page=\ipeFigSSchain]{locally-finite-figs.pdf}%
  \caption{The lattice $4$-SkewStrip with several chains of downward
    transposes and the join-irreducible elements circled}
  \label{fig:SSchain}
\end{subfigure}
\caption{}
\end{figure}

\section{Unique connected component of ideals} \label{sec:unique}

This section discusses a conjecture that was included in the
2\textsuperscript{\,nd} version of this paper but that has since been
disproved.

The examples we have been using, $\mbbZ\times\mbbZ$ and $\Bfin$, share
the property that their lattice of ideals $\scrI$ has multiple
connected components, one of which is isomorphic the lattice that
generates it.
However, our long-term goal is characterizing lattices which are useful
in constructing RSK algorithms.  Examples of these lattices are
$t$-SkewStrip in Fig.~\ref{fig:Plait-SS} and the other lattices listed in
\cite{Wor2026b}*{sec.~2.3}.
All of those lattices have the
property that their lattice of ideals $\scrI$ has at most three
components:
\begin{itemize}
\item one component that is $\im \phi$, isomorphic to the generating
  lattice $\scrL$;
\item if $\scrL$ has no $\onehat$, one component containing
  only $\scrP$; and
\item if $\scrL$ has no $\zerohat$, one component
containing only $\myemptyset$.
\end{itemize}
Looking at these examples suggested that the controlling characteristic
is whether $\scrL$ has finite width (contains no infinite
antichains).  This leads to:

\begin{conjecture} \label{conj:unique}
A locally-finite distributive lattice $\scrL$ has a lattice of ideals of prime
filters that contains at most three connected components, one of
which is $\im \phi$, isomorphic to the lattice;
perhaps one containing only the empty ideal; and
perhaps one containing only the ideal of all prime filters, iff
$\scrL$ has finite width.
\end{conjecture}

However, the AI editorial reviewing tool Refine provided a
counterexample:

\begin{example}
The lattice
$\mbbZ \times \mathbf{2}$ has no infinite antichains, but its lattice
of ideals of prime filters has three connected components, one isomorphic to
$\mbbZ \times \mathbf{2}$, the other two isomorphic to $\mathbf{2}$.
\end{example}

\section{Future directions} \label{sec:future}

Lattices that are locally finite and all of whose elements have finite
upward and downward covers are useful for constructing lattice
insertion algorithms.\cite{Wor2023c}  Thus it would be interesting to
determine the representation theorem for distributive lattices
that are both locally-finite and have finite covers.

It is likely that the representation and reconstruction
can be extended to be quasi-inverse contravariant functors.  That
requires a suitable definition of morphisms on pairs.

It would be interesting to extend the representation theorem to
modular lattices.

\section*{The development of this paper}

This analysis was developed in phases.  The first phase was developed
by looking at the problem directly.  Immediately evident is the
obstacle that some lattices do not have enough join-irreducible
elements to construct a representation.  Examining natural examples
from combinatorics suggests supplementing the
join-irreducible elements with the ``meet-closed [order] filters''
that are not the principal filters of join-irreducible elements.  (See
the note after Lem.~\ref{lem:ji-prime} for how this works out.)  From
there, Birkhoff's representation theorem suggests considering the
order ideals of the augmented set of join-irreducible elements.

Looking for references that discussed this problem, I found Nation's
\textit{Notes on Lattice Theory} \cite{Nat2017a}, which uses essentially
the same approach.  (Prof.\ Nation informs me that the discussion in
\cite{Nat2017a} is derived from Bob Dilworth's lattice theory course
at Caltech in 1970--71.)  \textit{Notes} informed me that
``meet-closed order filters'' were already known as ``lattice filters''.

The second phase was fleshing out this paper based on
\cite{Nat2017a} and aligning the terminology, including naming the map
from the distributive lattice to the lattice that represents it,
$\phi$.

Prof.\ Nation then provided valuable advice about aligning this
paper with the fuller development in Davey and Priestley
\cite{DavPriest1990a}, which is also aligned with Stone's original
development \cite{Stone1938a}.
The third phase was updating this paper to connect it with
\cite{DavPriest1990a}.

A fourth phase was applying the AI editorial tool Refine to the paper
and making improvements that it suggested.  One of these
was the refutation of Conj.~\ref{conj:unique}.
Another was rephrasing the representation as a duality.
This required formally defining the class of ``pairs''
(Def.~\ref{def:pair}) which required elucidating the concept of
``reachability'' (Def.~\ref{def:reach}).  The original formulation is
preserved as the ``alternative duality'' of
Sec.~\ref{sec:alternative}.

Asking ChatGPT about the conventional terminology for representation
dualities revealed that it is an interesting question whether a
representation is functorial, that is, whether it can be properly
extended to morphisms.  A quick check revealed that Birkhoff's
representation theorem is contravariantly functorial, leading to
investigation whether this duality is functorial.
However I decided to not pursue that question in order to finish this
paper sooner, especially because I would have to determine the correct
definition of morphisms on ``pairs''.

\vspace{3em}


\section*{References}

\begin{biblist}[\normalsize]*{labels={alphabetic}}


\DefineSimpleKey{bib}{identifier}{}
\DefineSimpleKey{bib}{location}{}
\DefineSimpleKey{bib}{primaryclass}{}
\gdef\MR#1{\relax\ifhmode\unskip\spacefactor3000 \space\fi
  \href{https://mathscinet.ams.org/mathscinet-getitem?mr=#1}{MR#1}}
\gdef\Zbl#1{\relax\ifhmode\unskip\spacefactor3000 \space\fi
  \href{https://zbmath.org/#1}{Zbl~#1}}
\gdef\GS#1{\relax\ifhmode\unskip\spacefactor3000 \space\fi
  \href{https://scholar.google.com/scholar?cluster=#1}{GS~#1}}
\gdef\ORCID#1{\relax\ifhmode\unskip\spacefactor3000 \space\fi
  \href{https://arxiv.org/search/?query=#1&searchtype=orcid&abstracts=hide&order=-announced_date_first&size=50}{ar$\chi$iv}
  \href{https://orcid.org/orcid-search/search?searchQuery=#1}{ORCID:#1}}
\gdef\myDOI#1{\relax\ifhmode\unskip\spacefactor3000 \space\fi
  \href{https://doi.org/#1}{DOI #1}}


\BibSpec{arXiv}{%
    +{}{\PrintAuthors}                  {author}
    +{,}{ \textit}                      {title}
    +{,} { \PrintTranslatorsC}          {translator}
    +{}{ \parenthesize}                 {date}
    +{,}{ arXiv }                       {identifier}
    +{,}{ primary class }               {primaryclass}
    +{,} { \myDOI}                      {doi}
    +{,} { available at \eprint}        {eprint}
    +{.} { }                            {note}
}

\BibSpec{article}{%
    +{}  {\PrintAuthors}                {author}
    +{,} { \textit}                     {title}
    +{.} { }                            {part}
    +{:} { \textit}                     {subtitle}
    +{,} { \PrintTranslatorsC}          {translator}
    +{,} { \PrintContributions}         {contribution}
    +{.} { \PrintPartials}              {partial}
    +{,} { }                            {journal}
    +{}  { \textbf}                     {volume}
    +{}  { \PrintDatePV}                {date}
    +{,} { \issuetext}                  {number}
    +{,} { \eprintpages}                {pages}
    +{,} { }                            {status}
    +{,} { \myDOI}                      {doi}
    +{,} { available at \eprint}        {eprint}
    +{}  { \parenthesize}               {language}
    +{}  { \PrintTranslation}           {translation}
    +{;} { \PrintReprint}               {reprint}
    +{.} { }                            {note}
    +{.} {}                             {transition}
    +{}  {\SentenceSpace \PrintReviews} {review}
}

\BibSpec{book}{%
    +{}  {\PrintPrimary}                {transition}
    +{.} { \emph}                       {title}
    +{.} { }                            {part}
    +{:} { \emph}                       {subtitle}
    +{.} { }                            {series}
    +{,} { \voltext}                    {volume}
    +{.} { Edited by \PrintNameList}    {editor}
    +{.} { Translated by \PrintNameList}{translator}
    +{.} { \PrintContributions}         {contribution}
    +{.} { }                            {publisher}
    +{.} { }                            {organization}
    +{,} { }                            {address}
    +{,} { \PrintEdition}               {edition}
    +{,} { \PrintDateB}                 {date}
    +{,} { \PrintISBNs}			{isbn}
    +{.} { }                            {note}
    +{.} {}                             {transition}
    +{.} { \PrintTranslation}           {translation}
    +{.} { Reprinted in \PrintReprint}  {reprint}
    +{.} {}                             {transition}
}

\BibSpec{partial}{%
    +{}  {}                             {part}
    +{:} { \textit}                     {subtitle}
    +{,} { \PrintContributions}         {contribution}
    +{,} { }                            {journal}
    +{}  { \textbf}                     {volume}
    +{}  { \PrintDatePV}                {date}
    +{,} { \issuetext}                  {number}
    +{,} { \eprintpages}                {pages}
    +{,} { \myDOI}                      {doi}
    +{,} { available at \eprint}        {eprint}
    +{.} { }                            {note}
}

\BibSpec{presentation}{%
    +{}{\PrintAuthors}                  {author}
    +{,}{ \textit}                      {title}
    +{,}{ }                             {date}
    +{,}{ }                             {location}
    +{,}{ }                             {series}
    +{,} { \myDOI}                      {doi}
    +{,} { available at \eprint}        {eprint}
    +{.} { }                            {note}
}

\BibSpec{misc}{%
    +{}  {\PrintPrimary}                {transition}
    +{,} { \textit}                     {title}
    +{.} { }                            {part}
    +{:} { \textit}                     {subtitle}
    +{,} { \PrintEdition}               {edition}
    +{}  { \PrintEditorsB}              {editor}
    +{,} { \PrintTranslatorsC}          {translator}
    +{,} { \PrintContributions}         {contribution}
    +{,} { }                            {organization}
    +{,} { }                            {address}
    +{,} { \PrintDateB}                 {date}
    +{,} { }                            {status}
    +{}  { \parenthesize}               {language}
    +{}  { \PrintTranslation}           {translation}
    +{;} { \PrintReprint}               {reprint}
    +{,} { \myDOI}                      {doi}
    +{,} { available at \eprint}        {eprint}
    +{.} { }                            {note}
    +{.} {}                             {transition}
    +{}  {\SentenceSpace \PrintReviews} {review}
}

\bib*{xref-Berk2024a}{book}{
  editor={Berkesch, Christine},
  editor={Brubaker, Benjamin},
  editor={Musiker, Gregg},
  editor={Pylyavskyy, Pavlo},
  editor={Reiner, Victor},
  title={Open problems in algebraic combinatorics},
  date={2024},
  publisher={AMS},
  address={Providence, RI, US},
  series={Proc.\ of Symposia in Pure Mathematics},
  volume={110},
  doi={10.1090/pspum/110},
  isbn={978-1470473334},
  eprint={https://www.ams.org/books/pspum/110/},
}

\bib*{xref-BogFreesKung1990a}{book}{
  title={The Dilworth theorems: Selected papers of Robert P.\ Dilworth},
  editor={Bogard, Kenneth P.},
  editor={Freese, Ralph S.},
  editor={Kung, Joseph P.~S.},
  date={1990},
  publisher={Springer},
  address={New York},
  series={Contemporary Mathematicians},
  doi={10.1007/978-1-4899-3558-8},
  isbn={978-0817634346},
}

\bib*{xref-Stan1999a}{book}{
  title={Enumerative Combinatorics, Volume 2},
  author={Stanley, Richard P.},
  date={1999},
  publisher={Cambridge University Press},
  address={Cambridge},
  series={Cambridge Studies in Advanced Mathematics},
  volume={62},
  isbn={978-1139639484},
}

\bib*{xref-Stant1990a}{book}{
  title={Invariant Theory and Tableaux},
  editor={Stanton, Dennis},
  publisher={Springer-Verlag},
  series={IMA Volumes in Math. and Its Appls.},
  volume={19},
  address={Berlin and New York},
  date={1990},
  isbn={978-0387971704},
}

\bib{Birk1937a}{article}{
  label={Birk1937},
  author={Birkhoff, Garrett},
  title={Rings of sets},
  journal={Duke Math.~J.},
  volume={3},
  date={1937},
  pages={443--454},
  review={\MR {1546000}},
  doi={10.1215/S0012-7094-37-00334-X},
  eprint={https://projecteuclid.org/journals/duke-mathematical-journal/volume-3/issue-3/Rings-of-sets/10.1215/S0012-7094-37-00334-X.short},
  note={\GS {10180976689018188837}},
}

\bib{Birk1967a}{book}{
  label={Birk1967},
  author={Birkhoff, Garrett},
  title={Lattice theory},
  edition={3},
  date={1967},
  publisher={American Mathematical Society},
  address={Providence},
  series={American Mathematical Society Colloquium Publications},
  volume={25},
  review={\MR {227053} \Zbl {0153.02501}},
  eprint={https://archive.org/details/latticetheory0000birk},
  note={Original edition 1940. \GS {10180976689018188837}},
}

\bib{CrawDil1973a}{book}{
  label={CrawDil1973},
  title={Algebraic Theory of Lattices},
  author={Crawley, J.~P.},
  author={Dilworth, R.~P.},
  date={1973},
  publisher={Prentice-Hall},
  address={Englewood Cliffs, N.~J.},
  review={\Zbl {0494.06001}},
  note={available at \url {https://archive.org/details/algebraictheoryo0000craw}},
}

\bib{DavPriest1990a}{book}{
  label={DavPriest1990},
  author={Davey, Brian A.},
  author={Priestley, Hilary Ann},
  title={Introduction to Lattices and Order},
  date={1990},
  edition={1},
  publisher={Cambridge Univ.\ Press},
  address={Cambridge, UK},
  series={Cambridge Math.\ Textbooks},
  review={\MR {1058437} \Zbl {0701.06001}},
  eprint={https://archive.org/details/introductiontola0000dave},
  note={\GS {17466900644257297586}},
}

\bib{Fom1994a}{article}{
  label={Fom1994},
  author={Fomin, Sergey V.},
  title={Duality of Graded Graphs},
  journal={J.~Algebr.\ Comb.},
  volume={3},
  date={1994},
  pages={357--404},
  review={\MR {1293822} \Zbl {0810.05005}},
  doi={10.1023/A:1022412010826},
  eprint={https://link.springer.com/content/pdf/10.1023/A:1022412010826.pdf},
  note={\GS {3401296478290474488}},
}

\bib{Nat2017a}{book}{
  label={Nat2017},
  title={Notes on Lattice Theory},
  author={Nation, James Bryant},
  date={2017},
  eprint={http://math.hawaii.edu/~jb/lattice2017.pdf},
  note={\GS {2514237753731863273} Links to individual chapters are in \url {https://math.hawaii.edu/~jb/books.html}},
}

\bib{Stan2012a}{book}{
  label={Stan2012},
  title={Enumerative Combinatorics, Volume 1},
  edition={2},
  author={Stanley, Richard P.},
  date={1986, 2012},
  publisher={Cambridge University Press},
  address={Cambridge},
  series={Cambridge Studies in Advanced Mathematics},
  volume={49},
  note={first edition 1986},
}

\bib{Stone1938a}{article}{
  label={Stone1938},
  author={Stone, Marshall Harvey},
  title={Topological representations of distributive lattices and Brouwerian logics},
  journal={\v {C}as.\ Mat.\ Fys.},
  volume={67},
  date={1938},
  pages={1--25},
  review={\Zbl {0018.00303}},
  eprint={https://eudml.org/doc/27235},
  note={\GS {13361446076336132270}},
}

\bib{Wor2023c}{arXiv}{
  label={Wor2023},
  author={Worley, Dale R.},
  title={On the combinatorics of tableaux --- Graphical representation of insertion algorithms},
  date={2023},
  identifier={2306.11140},
  primaryclass={math.CO},
  doi={10.48550/arXiv.2306.11140},
  eprint={https://arxiv.org/abs/2306.11140},
  note={\GS {2318818649542884742}},
}

\bib{Wor2026b}{arXiv}{
  label={Wor2026},
  author={Worley, Dale R.},
  title={On the combinatorics of tableaux --- Classification of lattices underlying Schensted correspondences},
  date={2026},
  identifier={2511.07611},
  primaryclass={math.CO},
  doi={10.48550/arXiv.2511.07611},
  eprint={https://arxiv.org/abs/2511.07611},
}

\end{biblist}
\end{document}